\documentclass[final]{siamltex}
\usepackage{amsfonts,amsmath,amssymb}
\usepackage{mathrsfs,mathtools}
\usepackage{enumerate}
\usepackage{xspace,mydef}
\usepackage{hyperref}
\usepackage{esint}
\usepackage{graphicx}
\DeclareMathAlphabet{\mathpzc}{OT1}{pzc}{m}{it}

\usepackage[usenames,dvipsnames]{color}
\newtheorem{remark}[theorem]{Remark}
\numberwithin{equation}{section}

\newcommand{\calL}{{\mathcal L}}
\DeclareMathOperator*{\esssup}{esssup}

\title{A PDE approach to space-time \\ fractional parabolic problems\thanks{RHN and EO are partially supported by NSF grants DMS-1109325 and DMS-1411808. EO is additionally supported by the Conicyt-Fulbright Fellowship Beca Igualdad de Oportunidades. AJS is partially supported by NSF grant DMS-1418784.}}

\author{Ricardo H.~Nochetto\thanks{Department of Mathematics and Institute for Physical
Science and Technology, University of Maryland, College Park, MD 20742, USA. \texttt{rhn@math.umd.edu}},
\and
Enrique Ot\'arola\thanks{Department of Mathematics, University of Maryland,
College Park, MD 20742, USA and Department of Mathematical Sciences, George Mason University, Fairfax, VA 22030, USA. \texttt{kike@math.umd.edu}.}
\and
Abner J.~Salgado\thanks{Department of Mathematics, University of Tennessee, Knoxville, TN 37996, USA.
\texttt{asalgad1@utk.edu}}}

\date{Draft version of \today.}

\begin{document}

\maketitle
\begin{abstract}
We study solution techniques for parabolic equations with fractional diffusion and Caputo fractional time derivative, the latter being discretized and analyzed in a general Hilbert space setting. 
The spatial fractional diffusion is realized as the Dirichlet-to-Neumann map for a nonuniformly elliptic problem posed on a semi-infinite cylinder in one more spatial dimension. We write our evolution problem as a quasi-stationary elliptic problem with a dynamic boundary condition. We propose and analyze an implicit fully-discrete scheme: first-degree tensor product finite elements in space and an implicit finite difference discretization in time. We prove stability and error estimates for this scheme.
\end{abstract}



\section{Introduction}
\label{sec:introduccion}

We are interested in the numerical approximation of an initial boundary value problem for a space-time fractional parabolic equation. Let $\Omega$ be an open and bounded subset of $\R^n$ ($n\ge1$), with boundary $\partial\Omega$. Given $s \in (0,1)$, $\gamma \in (0,1]$, a forcing function $f$, and an initial datum $\usf_0$, we seek
$\usf$ such that
\begin{equation}
\label{fractional_heat}
  \partial^{\gamma}_t \usf + \mathcal{L}^s \usf = f \ \text{ in } \Omega\times(0,T),
  \quad
  \usf(0) = \usf_0 \ \text{ in } \Omega,
  \quad
  \usf = 0 \ \text{on } \partial\Omega\times(0,T).
\end{equation}
Here $\calLs$, $s \in (0,1)$, is the fractional power of the second order elliptic operator
\begin{equation}
\label{second_order}
 \mathcal{L} w = - \DIV_{x'} (A \nabla_{x'} w ) + c w,
\end{equation}
where $0 \leq c \in L^\infty(\Omega)$ and $A \in C^{0,1}(\Omega,\GL(n,\R))$ is symmetric and positive definite.

The fractional derivative in time $\partial^{\gamma}_t$ for $\gamma \in (0,1)$ is understood as  \textit{the left-sided Caputo fractional derivative of order $\gamma$} with respect to $t$, which is defined by
\begin{equation}
\label{caputo}
\partial^{\gamma}_t \usf(x,t) := \frac{1}{\Gamma(1-\gamma)} \int_{0}^t \frac{1}{(t-r)^{\gamma}} \frac{\partial \usf(x,r)}{\partial r} \diff r,
\end{equation}
where $\Gamma$ is the Gamma function. For $\gamma = 1$, we consider the usual derivative $\partial_t$.

One of the main difficulties in the study of problem \eqref{fractional_heat} is the nonlocality of the fractional time derivative and the fractional space operator (see \cite{CT:10,CS:07,CDDS:11,fractional_book,Samko,ST:10}). A possible approach to overcome the nonlocality in space is given by the seminal result of Caffarelli and Silvestre in $\R^n$ \cite{CS:07} and its extensions to bounded domains \cite{CT:10,CDDS:11,ST:10}. Fractional powers of $\mathcal{L}$ can be realized as an operator that maps a Dirichlet boundary condition to a Neumann condition via an extension problem 
on $\C = \Omega \times (0,\infty)$. This extension is the following mixed boundary value problem (see \cite{CT:10,CS:07,CDDS:11,ST:10} for details):
\begin{equation}
\label{alpha_harm_L}
  \mathcal{L}\ue - \frac{\alpha}{y}\partial_{y}\ue - \partial_{yy}\ue = 0 \ \text{in } \C,
  \
  \ue = 0 \  \text{on } \partial_L \C,
  \
  \partial_{\nu}^{\alpha} \ue  = d_s f \ \text{on } \Omega \times \{0\},
\end{equation}
where $\partial_L \C= \partial \Omega \times [0,\infty)$ is the lateral boundary of $\C$, $\alpha = 1-2s \in (-1,1)$,
$d_s=2^\alpha \Gamma(1-s)/\Gamma(s)$ and the conormal exterior derivative of $\ue$ at $\Omega \times \{ 0 \}$ is
\begin{equation}
\label{def:lf}
\partial_{\nu}^{\alpha} \ue  = -\lim_{y \rightarrow 0^+} y^\alpha \ue_y.
\end{equation}

We will call $y$ the \emph{extended variable} and the dimension $n+1$ in $\R_+^{n+1}$ the \emph{extended dimension} of problem \eqref{alpha_harm_L}. The limit in \eqref{def:lf} must be understood in the sense of distributions; see \cite{CS:07,ST:10}. As noted in \cite{CT:10,CS:07,CDDS:11,ST:10}, we can relate the fractional powers of the operator $\calL$ with the Dirichlet-to-Neumann map of problem \eqref{alpha_harm_L}:
$
  d_s \calLs u = \partial_{\nu}^{\alpha} \ue
$
in $\Omega$. Notice that the differential operator in \eqref{alpha_harm_L} is
$
  -\DIV \left( y^{\alpha} \mathbf{A} \nabla \ue \right) + y^{\alpha} c\ue
$
where, for all $(x',y) \in \C $,  $\mathbf{A}(x',y) =  \textrm{diag} \{A(x'),1\} \in C^{0,1}(\C,\GL(n+1,\R))$.

The Caffarelli-Silvestre result has also been employed for the study of evolution equations with space fractional diffusion. For instance, by using this technique, H{\"o}lder estimates for the fractional heat equation were proved in \cite{SilvestreDFD}. We thus rewrite \eqref{fractional_heat} as a quasi-stationary elliptic problem with dynamic boundary condition:
\begin{equation}
\label{heat_alpha_extension}
\begin{dcases}
  -\DIV \left( y^{\alpha} \mathbf{A} \nabla \ve \right) + y^{\alpha} c\,\ve = 0 \ \textrm{in } \C\times(0,T), &
  \ve = 0 \ \textrm{on }\partial_L \C\times (0,T), \\
   d_s \partial_t^{\gamma} \ve +\partial_{\nu}^{\alpha} \ve = d_s f \ \textrm{on } (\Omega \times \{ 0\})\times(0,T), &
   \ve = \usf_0, \ \textrm{on }\Omega \times \{ 0\}, t=0.
\end{dcases}
\end{equation}

Before proceeding with the description and analysis of our method, let us give an overview of those advocated in the literature. The design of an efficient technique to treat numerically the left-sided Caputo fractional derivative of order $\gamma$ is not an easy task. The main difficulty is given by the nonlocality of the operator $\partial_t^{\gamma}$.
There are several approaches via finite differences, finite elements and spectral methods. For instance, a finite difference scheme is proposed and analyzed in \cite{LLX:11,LinXu:07} to deal with $\partial_t^{\gamma}$ and the so-called fractional cable equation. Semidiscrete finite element methods have been analyzed in \cite{BLZ:13} for \eqref{fractional_heat} with $\gamma \in (0,1)$ and $s=1$. Approaches via discontinuous Galerkin methods have been studied in \cite{MM:09,MM:13} for an alternative formulation of \eqref{fractional_heat} with $\gamma \in (0,2)$ and $s=1$. 
We refer to \cite[\S1]{MM:13} for an overview of the state of the art.

The finite difference scheme proposed in \cite{LLX:11,LinXu:07} has a consistency error $\mathcal{O}(\tau^{2-\gamma})$, where $\tau$ denotes the time step. This error estimate, however, requires a rather strong regularity assumption in time which is problematic; see \cite{M:10} and \S\ref{sub:discretization_beta}. Since $0 < \gamma
<1$, derivatives of the solution $\usf$ of \eqref{fractional_heat} with respect to $t$ are unbounded as $t \downarrow 0$.
In this work, we examine the singular behavior of $\partial_{t} \usf$ and $\partial_{tt} \usf$ when $t \downarrow 0$
and derive realistic time-regularity estimates for $u$; see also \cite{M:10,MM:13}. Using these refined results we analyze the truncation error and show discrete stability. The latter leads to an energy estimate for parabolic problems with
fractional time derivative in a general Hilbert space setting, written in terms of a fractional integral of a norm of $u$. We remark that H{\"o}lder regularity results for a parabolic equation with Caputo fractional time derivative have been recently establishedby Allen, Caffarelli and Vasseur in \cite{ACV:15}.

In prior work \cite{NOS} we used the Caffarelli-Silvestre extension to discretize the fractional space operator and obtained near-optimal error estimates in weighted Sobolev spaces for the extension.
We refer the reader to \cite{NOS} for a an overview of the existing numerical techniques
to solve elliptic problems involving fractional diffusion together with their advantages and disadvantages.
In this paper, we will adapt the approach developed in \cite{NOS} to the parabolic case.

We use the extension \eqref{heat_alpha_extension} to find the solution of \eqref{fractional_heat}: given $f$ and $\usf_0$, we solve \eqref{heat_alpha_extension}, thus obtaining a function $\ve: \C \times (0,T) \to \R$. Letting $\usf: \Omega \times (0,T) \to \R$ be $\usf(x',t) := \ve(x',0,t)$, we obtain  the solution of \eqref{fractional_heat}. The main objective of this work is to describe and analyze a fully discrete scheme for problem \eqref{heat_alpha_extension}.  We use implicit finite differences for time discretization \cite{LLX:11,LinXu:07}, and first degree tensor
product finite elements for space discretization.

The outline of this paper is as follows. In section~\ref{sec:Prelim} we introduce some terminology used throughout this work. We recall the definition of the fractional powers of elliptic operators via spectral theory in \S\ref{sub:fractional_L}, and in \S\ref{sub:CaffarelliSilvestre} we introduce the functional framework that is suitable to study problems \eqref{fractional_heat} and \eqref{heat_alpha_extension}. In \S\ref{sub:solution_representation},
we derive a representation for the solution of problem \eqref{alpha_harm_L}. We present regularity results in space and time in \S\ref{sub:space_regularity} and \S\ref{sub:time_regularity}, respectively. The time discretization of problem \eqref{fractional_heat} is analyzed in section~\ref{sec:time_discretization}: the case $\gamma = 1$ is discretized by the standard backward Euler scheme whereas, for $\gamma \in (0,1)$, we consider the finite difference approximation of \cite{LLX:11,LinXu:07}. For both cases we derive stability results and a novel energy estimate
for parabolic problems with fractional time derivative in a general Hilbert space setting. We discuss error estimates for semi-discrete schemes in \S\ref{sub:ErrTimedisc}. The space discretization of problem \eqref{heat_alpha_extension}
begins in section~\ref{sec:space_discretization}: in \S\ref{sub:truncation}, we introduce a truncation of the domain $\C$ and study some properties of the solution of a truncated problem; in \S\ref{sec:sub:space_discretization} we present
the finite element approximation to the solution of \eqref{heat_alpha_extension} in a bounded domain and in \S\ref{sec:sub:elliptic_projector} we study a weighted elliptic projector and its properties.
In section~\ref{sec:fully_scheme}, we deal with fully discrete schemes and derive error estimates for all $\gamma \in (0,1]$ and $s \in (0,1)$. 

\section{Solution representation and regularity}
\label{sec:Prelim}
Throughout this work $\Omega$ is an open, bounded and connected subset of $\R^n$, $n\geq1$, with polyhedral boundary $\partial\Omega$. We define the semi-infinite cylinder and its lateral boundary, respectively, by $\C = \Omega \times (0,\infty)$ and $\partial_L \C  = \partial \Omega \times [0,\infty)$. Given $\Y>0$, we define the truncated cylinder
$
  \C_\Y = \Omega \times (0,\Y)
$
and $\partial_L\C_\Y$ accordingly. If $x\in \R^{n+1}$, we write
$
 x = (x',y),
$
with $x' \in \R^n$ and $y\in\R$. If $\Xcal$ is a normed space, $\Xcal'$ denotes its dual and $\|\cdot\|_{\Xcal}$ its norm.
The relation $a \lesssim b$ means $a \leq cb$, with a nonessential constant $c$ that might change at each occurrence.

If $T >0$ and $\phi: \D \times(0,T) \to \R$, with $\D$ a domain in $\R^{N}$ ($N \geq 1$), we consider $\phi$ as a function of $t$ with values in a Banach space $\Xcal$,
$
 \phi:(0,T) \ni t \mapsto  \phi(t) \equiv \phi(\cdot,t) \in \Xcal
$.
For $1 \leq p \leq \infty$, $L^p( 0,T; \Xcal )$
is the space of $\Xcal$-valued functions whose norm in $\Xcal$ is in $L^p(0,T)$. This is a Banach space for the norm
\[
  \| \phi \|_{L^p( 0,T;\Xcal)} = \left( \int_0^T \| \phi(t) \|^p_\Xcal \right)^{\hspace{-0.1cm}\tfrac{1}{p}} 
  , \quad 1 \leq p < \infty, \quad
  \| \phi \|_{L^\infty( 0,T;\Xcal)} = \esssup_{t \in (0,T)} \| \phi(t) \|_\Xcal.
\]

In \eqref{fractional_heat}, $\partial_t^\gamma$ denotes the left-sided Caputo fractional derivative \eqref{caputo}. There are three, not equivalent, definitions of fractional derivatives: Riemann-Liouville, Caputo and Gr\"unwald-Letnikov. For their definitions and properties see \cite{fractional_book,Samko}.

\subsection{Fractional integrals}
\label{sub:fractional_integral}

Given a function $g \in L^1(0,T)$, the left Riemann-Liouville fractional integral $I^{\sigma} g$ of order $\sigma>0$ is defined by \cite{fractional_book,Samko}:
\begin{equation}
\label{fractional_integral}
(I^{\sigma} g)(t) = \frac{1}{\Gamma(\sigma)} \int_{0}^t \frac{g(r)}{(t-r)^{1-\sigma}} \diff r;
\end{equation}
note that $\partial_t^\gamma g(t) = (I^{1-\gamma} \partial_t g)(t)$ for all $g \in W_1^1(0,T)$. Young's inequality for convolutions immediately yields the following result.
\begin{lemma}[continuity]
\label{le:continuity}
If $g \in L^2(0,T)$ and $\phi \in L^1(0,T)$, then the operator
\[
 g \mapsto \Phi, \qquad \Phi(t) = \phi \star g (t) = \int_0^t \phi(t-r) g(r) \diff r
\]
is continuous from $L^2(0,T)$ into itself and 
$
 \| \Phi \|_{L^2(0,T)} \leq \| \phi \|_{L^1(0,T)} \| g \|_{L^2(0,T)}.
$
\end{lemma}

\begin{corollary}[continuity of $I^{\sigma}$]
\label{co:continuity}
For any $\sigma > 0$, the left Riemann-Liouville fractional integral $I^{\sigma} g$ is continuous from
$L^2(0,T)$ into itself and
\[\|I^{\sigma} g\|_{L^2(0,T)}\le \frac{T^\sigma}{\Gamma(\sigma+1)}
\|g\|_{L^2(0,T)} \quad\forall g\in L^2(0,T).
\]
\end{corollary}

\subsection{Fractional powers of general second order elliptic operators}
\label{sub:fractional_L}

The operator $\calL^{-1}: L^2(\Omega)\to L^2(\Omega)$, which solves $\mathcal{L} w  = f$ in $\Omega$ and $w = 0$ on $\partial \Omega$ is compact, symmetric and positive, so its spectrum $\{\lambda_k^{-1} \}_{k\in \mathbb N}$ is discrete, real, positive and accumulates at zero. Moreover, the eigenfunctions $\{\varphi_k\}_{k \in \mathbb{N}}$
\begin{equation}
   \label{eigenvalue_problem_L}
    \mathcal{L} \varphi_k = \lambda_k \varphi_k  \text{ in } \Omega,
    \qquad
    \varphi_k = 0 \text{ on } \partial\Omega, \qquad k \in \mathbb{N}
\end{equation}
form an orthonormal basis of $L^2(\Omega)$. Fractional powers of $\mathcal L$ can be defined by
\begin{equation}
  \label{def:second_frac}
  \mathcal{L}^s w  := \sum_{k=1}^\infty \lambda_k^{s} w_k \varphi_k, \qquad w \in C_0^{\infty}(\Omega), \qquad s \in (0,1),
\end{equation} 
where $w_k = \int_{\Omega} w \varphi_k $. By density we extend this definition to
\begin{equation}
\label{def:Hs}
  \Hs = \left\{ w = \sum_{k=1}^\infty w_k \varphi_k: 
  \sum_{k=1}^{\infty} \lambda_k^s w_k^2 < \infty \right\} = [H^1_0(\Omega),L^2(\Omega)]_{1-s}; 
\end{equation}
see \cite{NOS} for details. For $ s \in (0,1)$ we denote by $\Hsd$ the dual space of $\Hs$.

\subsection{The Caffarelli-Silvestre extension problem}
\label{sub:CaffarelliSilvestre}

The Caffarelli-Silvestre result \cite{CT:10,CS:07,CDDS:11,ST:10}, requires to deal with a nonuniformly elliptic equation. Let $D \subset \R^{n+1}$ be open and define $L^2(|y|^{\alpha},D)$ as the Lebesgue space for the measure $|y|^\alpha \diff x$.

Define also $ H^1(|y|^{\alpha},D) := \{ w \in L^2(|y|^{\alpha},D): | \nabla w | \in L^2(|y|^{\alpha},D) \}$, with norm
\begin{equation}
\label{wH1norm}
\| w \|_{H^1(|y|^{\alpha},D)} =
\left(  \| w \|^2_{L^2(|y|^{\alpha},D)} + \| \nabla w \|^2_{L^2(|y|^{\alpha},D)} \right)^{\frac{1}{2}}.
\end{equation}
Since $\alpha \in (-1,1)$, $|y|^\alpha$ belongs to the Muckenhoupt class $A_2(\R^{n+1})$; see \cite{GU,Turesson}. This implies that $H^1(|y|^{\alpha},D)$ is Hilbert and $C^{\infty}(\D) \cap H^1(|y|^{\alpha},D)$ is dense in $H^1(|y|^{\alpha},D)$ (cf.~\cite[Proposition 2.1.2, Corollary 2.1.6]{Turesson} and \cite[Theorem~1]{GU}).

To study problem \eqref{heat_alpha_extension} we define the weighted Sobolev space
\begin{equation}
  \label{HL10}
  \HL(y^{\alpha},\C) := \left\{ w \in H^1(y^\alpha,\C): w = 0 \textrm{ on } \partial_L \C\right\}.
\end{equation}
As \cite[(2.21)]{NOS} shows, the following \emph{weighted Poincar\'e inequality} holds:
\begin{equation}
\label{Poincare_ineq}
\| w \|_{L^2(y^{\alpha},\C)} \lesssim \| \nabla v \|_{L^2(y^{\alpha},\C)},
\quad \forall w \in \HL(y^{\alpha},\C).
\end{equation}
Then, the seminorm on $\HL(y^{\alpha},\C)$ is equivalent to the norm \eqref{wH1norm}. For $w \in H^1(y^{\alpha},\C)$ $\tr w$ denotes its trace onto $\Omega \times \{ 0 \}$. We recall (\cite[Prop.~2.5]{NOS} and \cite[Prop.~2.1]{CDDS:11})
\begin{equation}
\label{Trace_estimate}
\tr \HL(y^{\alpha},\C) = \Hs,
\qquad
  \|\tr w\|_{\Hs} \leq C_{\tr} \| w \|_{\HLn(y^{\alpha},\C)}.
\end{equation}

The Caffarelli-Silvestre extension result \cite{CS:07,ST:10} then reads: If $u \in \Hs$ solves $\calL^s u = f$ in $\Omega$ and $\ue \in \HL(y^\alpha,\C)$ solves \eqref{alpha_harm_L}, then $\tr \ue = u$.

To write the appropriate Caffarelli-Silvestre extension for problem \eqref{heat_alpha_extension}, we define:
\begin{equation}\label{spaces}
\begin{aligned}
  \mathbb{W} &:= \{ w \in L^{\infty}(0,T;L^2(\Omega)) \cap L^{2}(0,T;\Hs): \partial_t^{\gamma} 
    w \in L^2(0,T;\Hsd)\}, \\
  \mathbb{V} &:= \{ w \in L^{2}(0,T;\HL(y^{\alpha},\C)): \partial_t^{\gamma} \tr w  \in L^2(0,T;\Hsd)\}.  
\end{aligned}
\end{equation}
Thus, given $f \in L^2(0,T;\Hsd)$, a function $\usf \in \mathbb{W}$ solves \eqref{fractional_heat} if and only if the
harmonic extension $\ve \in \mathbb{V}$ solves \eqref{heat_alpha_extension}. A weak formulation of \eqref{heat_alpha_extension} reads: Find $\ve \in \mathbb{V}$ such that $\tr \ve(0) = \usf_0$ and, for a.e.~$t \in (0,T)$, 
\begin{equation}
\label{heat_harmonic_extension_weak}
\langle \tr  \partial_t^{\gamma} \ve, \tr \phi \rangle+ a(\ve,\phi) = \langle f, \tr \phi \rangle \quad \forall \phi \in \HL(y^{\alpha},\C),
\end{equation}
where $\langle \cdot, \cdot \rangle$ is the duality pairing between $\Hs$ and $\Hsd$ and
\begin{equation}
\label{a}
a(w,\phi) :=   \frac{1}{d_s}\int_{\C} {y^{\alpha}\mathbf{A}(x)} \nabla w \cdot \nabla \phi + y^{\alpha} c(x') w \phi. 
\end{equation}

\begin{remark}[equivalent seminorm]
\label{rem:equivalent} \rm
The regularity of $A$ and $c$ and \eqref{Poincare_ineq} imply that $a$, defined in \eqref{a}, is bounded and coercive in $\HL(y^\alpha,\C)$. In what follows we shall use repeatedly that $a(w,w)^{1/2}$ is an equivalent norm to $| \cdot |_{H^1(y^\alpha,\C)}$ in $\HL(y^\alpha,\C)$.
\end{remark}

\begin{remark}[dynamic boundary condition]
\label{rem:dynamical} \rm
Problem \eqref{heat_harmonic_extension_weak} is an elliptic problem with a dynamic boundary condition: $\partial_{\nu}^{\alpha} \ve = f - \tr  \partial_t^{\gamma} \ve$ on $\Omega \times \{0\}$. Consequently, its analysis is slightly different from the standard theory for parabolic equations.
\end{remark}

\begin{remark}[initial datum]
\label{rem:initial_data} \rm
The initial datum $\usf_0$ of problem \eqref{fractional_heat} determines only $\ve(0)$ on $\Omega \times \{ 0\}$ in a trace sense. However, in the subsequent analysis it is necessary to consider its extension to the whole 
cylinder $\C$. Thus, we define $\ve(0)$ to be the solution of problem \eqref{alpha_harm_L} with the Neumann condition replaced by the Dirichlet condition $\tr \ve = \usf_0$. References \cite{CT:10,CDDS:11} provide the estimate 
$\| \ve(0) \|_{\HLn(y^{\alpha},\C)} \lesssim \| \usf_0 \|_{\Hs}$.
\end{remark}

\subsection{Solution representation}
\label{sub:solution_representation}
Using the eigenpairs $\{ \lambda_k, \varphi_k\}$ we deduce that if $\usf(x',t)=\sum_k \usf_k(t) \varphi_k(x')$ solves \eqref{fractional_heat} then $\ve$, solution of \eqref{heat_alpha_extension}, can be written as
\begin{equation}
\label{exactforms}
  \ve(x,t) = \sum_{k=1}^\infty \usf_k(t) \varphi_k(x') \psi_k(y),
\end{equation}
where $\psi_k$ solves
\begin{equation}
\label{psik}
\psi_k'' + \alpha y^{-1} \psi_k' - \lambda_k \psi_k = 0, \quad
\psi_k(0) = 1, \quad \psi_k(y) \to 0, \ y \to \infty.
\end{equation}
If $s=\sr$, then $\psi_k(y) = e^{-\sqrt{\lambda_k}y}$. For $s \in (0,1)\setminus\{\srn\}$ we have that if $c_s = \tfrac{2^{1-s}}{\Gamma(s)}$, then 
$
\psi_k(y) = c_s \left(\sqrt{\lambda_k}y\right)^s K_s(\sqrt{\lambda_k} y),
$
where $K_s$ denotes the modified Bessel function of the second kind; see \cite{CDDS:11,NOS}. For $s \in (0,1)$, we have \cite{NOS}
\begin{equation}
 \label{asym-psi_k-p}
 \lim_{y \downarrow 0^{+} } \frac{y^{\alpha}\psi_{k}'(y)}{d_s \lambda_k^s } = - 1,
\end{equation}
and, for  $a, b \in \R^+$, $a<b$
\begin{equation}
 \label{besselenergy}
 \int_{a}^{b} y^{\alpha} \left( \lambda_k \psi_k(y)^2 + \psi_k'(y)^2 \right) \diff y
 = \left. y^{\alpha} \psi_k(y)\psi_k'(y) \right|_{a}^{b}.
\end{equation}
The boundary condition of \eqref{heat_alpha_extension}, in conjunction with \eqref{def:lf}, and \eqref{exactforms}--\eqref{asym-psi_k-p} imply
\begin{equation}
\label{eq:aux_key}
    d_s f = -\lim_{y \downarrow 0 } y^{\alpha} \ve_y + d_s \tr \partial_t^{\gamma}\ve
    = d_s \sum_{k=1}^{\infty} \varphi_{k} \left( \lambda_k^s \usf_k + \partial_t^{\gamma} \usf_k \right),
\end{equation}
which, in turn, since $\usf|_{t=0} = \usf_0$, yields the fractional initial value problem for $\usf_k$
\begin{equation} 
 \label{uk}
    \partial^{\gamma}_t \usf_k(t) + \lambda_k^s \usf_k(t) = f_k(t), \ t>0, \qquad
     \usf_k(0) = \usf_{0,k},
\end{equation}
with $\usf_{0,k} = (\usf_0,\varphi_k)_{L^2(\Omega)}$, and $f_{k} = \langle f,\varphi_k\rangle$. The theory of fractional ordinary differential equations \cite{fractional_book,Samko} gives a unique function $\usf_k$ satisfying problem \eqref{uk}. In addition, using \eqref{exactforms} and \eqref{psik}, we obtain
\[
  \ve(x',0,t) = \sum_{k=1}^{\infty} \usf_k(t) \varphi_k(x') \psi_k(0) =
  \sum_{k=1}^{\infty} \usf_k(t) \varphi_k(x') = \usf(x',t).
\]
Finally, Remark~\ref{rem:equivalent} together with formulas \eqref{asym-psi_k-p} and \eqref{besselenergy} imply
\begin{equation}
\| \nabla \ve(t) \|^2_{L^2(y^{\alpha},\C)} \lesssim \sum_{k=1}^{\infty} \usf_k(t)^2 \int_{0}^{\infty} 
y^{\alpha} \left( \lambda_k \psi_k(y)^2 + \psi_k'(y)^2 \right)
= d_s \| \usf(t)\|^2_{\Hs}
\label{norm=} 
\end{equation}
for a.e.~$t \in (0,T)$. We now turn our attention to the solution of problem \eqref{uk}.

\subsubsection{Case $\gamma = 1$: The exponential function}
\label{subsubsec:exponential_function}
If $\gamma = 1$, then \eqref{uk} reduces to a first-order initial value problem. We define $E(t) w = \sum_{k=1}^{\infty} e^{-\lambda_k^s t}(w, \varphi_k)_{L^2(\Omega)} \varphi_k$, which is the solution operator of \eqref{fractional_heat} with $f \equiv 0$. By Duhamel's principle, the solution of problem \eqref{fractional_heat} is
$\usf(x',t) = E(t) \usf_0 + \int_{0}^t E(t-r) f(x',r) \diff r$.

\subsubsection{Case $\gamma \in (0,1)$: The Mittag-Leffler function}
\label{subsubsec:ML_function}
For $\gamma > 0$ and $\mu \in \R$, we define the Mittag Leffler function $E_{\gamma,\mu}(z)$ as
\begin{equation}
\label{ML}
 E_{\gamma,\mu}(z):= \sum_{k=0}^{\infty} \frac{z^k}{\Gamma( \gamma k + \mu )}, \quad z \in \mathbb{C};
\end{equation}
see \cite{fractional_book,Samko}. For $\lambda,\gamma,t \in \R^+$, we have \cite[Lemma 2.23]{fractional_book}
\begin{equation}
\label{dEbeta1}
\partial_t^{\gamma} E_{\gamma,1}(-\lambda t^{\gamma}) = -\lambda E_{\gamma,1}(-\lambda t^{\gamma}).
\end{equation}
If $\gamma \in (0,2)$, $\mu \in \R$, $\pi \gamma/2 < \delta < \min\{ \pi, \pi \gamma \}$ and $\delta \leq |\arg(z) | \leq \pi$, then \cite[\S 1.8]{fractional_book}
\begin{equation}
\label{ML_estimate}
  (1 + |z|)^{-1} |E_{\gamma,\mu}(z)| \lesssim 1.
\end{equation}
Following \cite{Sakayama} we construct the solution to \eqref{fractional_heat}. The solution operator for $f  \equiv 0$ is
\begin{equation}
\label{G}
 G_{\gamma}(t) w = \sum_{k=1}^{\infty} E_{\gamma,1}(-\lambda^s_k t^{\gamma}) w_k \varphi_k.
\end{equation}
which follows from \eqref{dEbeta1}; see also \cite[(2.3)]{BLZ:13} and \cite[(2.6)]{M:10} for the particular case $s=1$. If $f \neq 0$ and $\usf_0 \equiv 0$, we also define the operator
\begin{equation}
\label{barG}
 F_{\gamma}(t) w = \sum_{k=1}^{\infty} t^{\gamma-1}E_{\gamma,\gamma}(-\lambda^s_k t^{\gamma}) w_k \varphi_k.
\end{equation}
Using these operators, we have (\cite[(2.4)]{BLZ:13} and \cite[Theorem 2.2]{Sakayama} for $s=1$)
\begin{equation}
\label{u_exact}
\usf(x',t) = G_{\gamma}(t) \usf_0 + \int_0^t F_{\gamma}(t-r)f(x',r)\diff r; 
\end{equation}
These considerations yield existence and uniqueness for solutions of \eqref{fractional_heat} and \eqref{heat_alpha_extension}. We refer to~\S \ref{sec:time_discretization} for energy estimates 
(see also \cite{Sakayama}).

\begin{theorem}[existence and uniqueness]
\label{thm:exis_uniq}
Given $s \in (0,1)$, $\gamma \in (0,1]$, $f \in L^2(0,T;\Hsd)$ and $\usf_0 \in L^2(\Omega)$, problems
\eqref{fractional_heat} and \eqref{heat_alpha_extension} have a unique solution.
\end{theorem}
\begin{proof}
Existence and uniqueness of problem \eqref{fractional_heat} can be obtained modifying the spectral decomposition 
approach studied in \cite{Sakayama} based on the solution representation \eqref{u_exact}; see 
\cite[Theorems 2.1 and 2.2]{Sakayama}.
Similar arguments apply to conclude the well-posedness of problem \eqref{heat_alpha_extension}. For brevity, we 
leave the details to the reader.
\end{proof}

\subsection{Regularity}
\label{sub:Regularity}
Let us now discuss the space and time regularity of $\ve$. In what follows we tacitly assume that $\Omega$ is such that
\begin{equation}
\label{eq:H2reg}
  \| w \|_{H^2(\Omega)} \lesssim \| \mathcal{L} w \|_{L^2(\Omega)}, \quad \forall w \in H^2(\Omega) \cap H^1_0(\Omega).
\end{equation}

\subsubsection{Space regularity}
\label{sub:space_regularity}


The regularity in space of $\ve$ is described below.

\begin{theorem}[space regularity]
\label{TH:regularity}
Let $\ve \in \HL(y^{\alpha},\C)$ solve \eqref{heat_alpha_extension}. For $s \in (0,1) \setminus \{\sr\}$ and $\gamma = 1$, we have
\begin{align}
\label{reginx}
\| \nabla \nabla_{x'} \ve \|^2_{L^2(0,T;L^2(y^{\alpha},\C))} &\lesssim T \| \usf_0\|^2_{\mathbb{H}^{1+s}(\Omega)}+
\| f \|_{L^2(0,T;\Ws)}^2,\\
\label{reginy}
\| \ve_{yy} \|_{L^2(0,T;L^2(y^{\beta},\C))}^2 & \lesssim T \| \usf_0\|^2_{\mathbb{H}^{2s}(\Omega)}+ 
\| f \|_{L^2(0,T;L^2(\Omega))}^2,
\end{align}
with $\beta>2\alpha+1$. Let $0 < \mu \ll 1$ be arbitrary. For $s \in (0,1) \setminus \{\sr\}$ and $\gamma \in (0,1)$, we have
\begin{align}
\label{reginxgamma}
\| \nabla \nabla_{x'} \ve \|^2_{L^2(0,T;L^2(y^{\alpha},\C))} &\lesssim T \| \usf_0\|^2_{\mathbb{H}^{1+s}(\Omega)}+
T^{2\gamma\mu}\| f \|_{L^2(0,T;\mathbb{H}^{1-(1-2\mu)s}(\Omega))}^2,\\
\label{reginygamma}
\| \ve_{yy} \|_{L^2(0,T;L^2(y^{\beta},\C))}^2 & \lesssim T \| \usf_0\|^2_{\mathbb{H}^{2s}(\Omega)}+
T^{2\gamma\mu}\| f \|_{L^2(0,T;\mathbb{H}^{2\mu s}(\Omega))}^2.
\end{align}
For $s = \sr$, we have  
\begin{align}\label{H2-gamma=1}
\| \ve \|^2_{L^2(0,T;H^2(\C))} &\lesssim T \| \usf_0\|^2_{\mathbb{H}^{3/2}(\Omega)} + 
\| f \|_{L^2(0,T;\mathbb{H}^{1/2}(\Omega))}^2, \quad \gamma =1,
\\
\label{H2-gamma<1}
\| \ve \|^2_{L^2(0,T;H^2(\C))} &\lesssim T \| \usf_0\|^2_{\mathbb{H}^{3/2}(\Omega)}+ T^{2\gamma\mu}\| f \|_{L^2(0,T;\mathbb{H}^{1/2-\mu}(\Omega))}^2, \quad \gamma \in (0,1).
\end{align}
\end{theorem}
\begin{proof}
We proceed in several steps using the representation formula \eqref{exactforms}.

\noindent \boxed{1}
Case $s \in (0,1) \setminus \{\frac{1}{2}\}$. Since $\{ \varphi_k\}_{k\in \mathbb N}$  satisfies \eqref{eigenvalue_problem_L} and $\int_{0}^{\infty} y^{\beta} |\psi_k''(y)|^2 \lesssim \lambda_k^{3/2-\beta/2} \leq
\lambda_k^{2s}$  \cite[Theorem 2.7]{NOS}, we obtain
\begin{equation}
\label{V_yy}
\| \ve_{yy}(\cdot,t) \|^2_{L^2(y^\beta,\C)} = \sum_{k=1}^{\infty} |\usf_k(t)|^2 \int_{0}^{\infty} y^{\beta} |\psi_k''(y)|^2 \diff y \lesssim \sum_{k=1}^{\infty} \lambda_k^{2s} |\usf_k(t)|^2.
\end{equation}
On the other hand, if 
$
\mathcal{D}(\ve(\cdot,t)) := \int_\C y^\alpha\big( |\calL \ve|^2 +A \nabla_{x'}\partial_y \ve
\cdot \nabla_{x'}\partial_y \ve + c |\partial_y \ve|^2\big) \diff x' \diff y
$, then we realize that 
$
\|\nabla \nabla_{x'}\ve(\cdot,t)\|_{L^2(y^\alpha,\C)}^{2} \lesssim \mathcal{D}(\ve(\cdot,t)).
$
We exploit $\int_{0}^{\infty} y^{\alpha}(\lambda_k \psi_k(y)^2 + \psi_k'(y)^2) \diff y \lesssim \lambda_{k}^s$ (see \cite[Theorem 2.7]{NOS}), to arrive at
\begin{equation}
\label{D}
\mathcal{D}(\ve(\cdot,t)) \lesssim \sum_{k=1}^{\infty} \lambda_k |\usf_k(t)|^2\int_{0}^{\infty} y^{\alpha}(\lambda_k \psi_k(y)^2 + \psi_k'(y)^2)
\lesssim \sum_{k=1}^{\infty} \lambda_k^{1+s} |\usf_k(t)|^2.
\end{equation}
We thus need to estimate $\|\usf_k\|_{L^2(0,T)}$. We distinguish between $\gamma=1$ and $\gamma<1$.

\noindent 
\boxed{2} Case $\gamma = 1$. We recall the representation formula
$\usf_k(t) = e^{-\lambda_k^s t} \usf_{0,k} + \int_{0}^t e^{-\lambda_k^s r} f_k(t-r)\diff r$
from  \S \ref{subsubsec:exponential_function},
where $\usf_{0,k} = (\usf_0,\varphi_k)_{L^2(\Omega)}$. Consequently,
we get
\begin{equation}\label{L2-uk}
\| \usf_k \|^2_{L^2(0,T)} \lesssim 
\int_{0}^T \left( \usf^2_{0,k} + 
(e^{-\lambda_k^s t} \star f_k)(t)^2 \right) \diff t
\end{equation}
and $\| e^{-\lambda_k^s t} \star f_k\|_{L^2(0,T)} \leq \lambda_{k}^{-s}\| f_k\|_{L^2(0,T)}$
according to Lemma~\ref{le:continuity}.
This, in conjunction with \eqref{V_yy} and \eqref{D}, implies
\eqref{reginx} and \eqref{reginy}.

\noindent 
\boxed{3} Case $\gamma \in (0,1)$. We recall the representation formula $\usf_k(t) = E_{\gamma,1}( -\lambda_k^s t^{\gamma} )\usf_{0,k} + \int_0^t r^{\gamma-1}E_{\gamma,\gamma}( -\lambda_k^s r^{\gamma})f_k(t-r)\diff r$ from \S \ref{subsubsec:ML_function}. Using \eqref{ML_estimate}, we deduce
\begin{equation}\label{log-term}
\|r^{\gamma-1}E_{\gamma,\gamma}( -\lambda_k^s r^{\gamma} ) \|_{L^1(0,T)} \lesssim \lambda_k^{-s}\log (1+\lambda_k^s T^{\gamma}).
\end{equation}
This, together with the preceding expression for $\usf_k(t)$ and Lemma \ref{le:continuity}, gives
\begin{equation}\label{L2-uk-gamma}
\| \usf_k\|_{L^2(0,T)}^2 \lesssim T \usf_{0,k}^2 + \lambda_k^{-2s}\log^2 (1+\lambda_k^s T^{\gamma}) \|f_k\|_{L^2(0,T)}^2.
\end{equation}
Inserting this into \eqref{V_yy} and \eqref{D}, and using that $\log (1+z)\lesssim z^\mu$ for all $z\ge 0$ and $\mu>0$,
yields the asserted estimates  \eqref{reginxgamma} and \eqref{reginygamma}.

\noindent \boxed{4}
Case $s = \frac{1}{2}$. Since $\| \ve(\cdot,t) \|^2_{H^2(\C)} \lesssim \sum_{k=1}^{\infty} \lambda_k^{\frac{3}{2}} |\usf_k(t)|^2$, applying \eqref{L2-uk} and \eqref{L2-uk-gamma} leads to \eqref{H2-gamma=1} and \eqref{H2-gamma<1}, respectively.
\end{proof}

We summarize the conclusion of Theorem~\ref{TH:regularity} as follows. Define, for $\beta > 1+2\alpha$,
\begin{equation}
\label{C(w)}
\mathcal{S}(w(\cdot,t)) := \|  \nabla \nabla_{x'} w(\cdot,t) \|_{L^2(y^\alpha,\C)} + \|  \partial_{yy} w(\cdot,t) \|_{L^2(y^\beta,\C)},
\end{equation}
and
\begin{equation}
\label{eq:defoffrakR}
  \mathcal{R}^2(\usf_0,f) = 
  \begin{dcases}
    T \| \usf_0\|^2_{\mathbb{H}^{1+s}(\Omega)}+ \| f \|_{L^2(0,T;\Ws)}^2 & \gamma = 1, \\
    T \| \usf_0\|^2_{\mathbb{H}^{1+s}(\Omega)}+
  T^{2\gamma\mu}\| f \|_{L^2(0,T;\mathbb{H}^{1-(1-2\mu)s}(\Omega))}^2 & \gamma \in (0,1).
  \end{dcases}
\end{equation}
for any $\mu > 0$. Then, for $s \in (0,1)$ and $\gamma \in (0,1]$, we have
\begin{equation}
\label{eq:S_in_spacetime}
  \| \mathcal{S}(\ve) \|_{L^2(0,T)} \lesssim \mathcal{R}(\usf_0,f).
\end{equation}
\subsubsection{Time regularity}
\label{sub:time_regularity}
We now focus on the regularity in time. For $\gamma = 1$, we could demand sufficient regularity (in time) of the right-hand side along with compatibility conditions for the initial datum $\usf_0$. We express this as
\begin{equation}
\label{eq:reg_time_beta01}
  \tr \partial_{tt}\ve \in L^2(0,T;\Hsd).
\end{equation}

For $\gamma \in (0,1)$, \eqref{eq:reg_time_beta01} is inconsistent with \eqref{u_exact}. In fact, properties of the Mittag-Leffler function and \eqref{u_exact} for $f=0$ show that \eqref{eq:reg_time_beta01} never holds if $\usf_0 \neq 0$ because  
\begin{equation}
\label{eq:asympt}
  \usf(x',t) = G_\gamma(t) \usf_0(x') = \left( 1 - \frac{t^\gamma}{\Gamma(1+\gamma)}\calL^s +\mathcal{O}(t^{2\gamma})  \right)\usf_0(x') \quad \textrm{as }t \downarrow 0.
\end{equation}
We see that derivatives of $\usf$ with respect to $t$ are unbounded as $t \downarrow 0$ for $\gamma \in (0,1)$ and, in particular, $\partial_{tt} \usf (x',t) \approx t^{\gamma-2} \calL^s \usf_0(x') \notin L^2(0,T;\Hsd)$. However
\[
  \int_{0^{+}} t^\sigma \|\partial_{tt} \usf (\cdot,t) \|_{\Hsd}^2 \diff t \lesssim \| \usf_0 \|_{\Hs}^2 \int_{0^{+}} t^{\sigma+2\gamma-4} \diff t
\]
is finite provided $\sigma > 3-2\gamma$. For this reason, when $\gamma \in (0,1)$, we assume
\begin{equation}\label{reg-time-gamma<1}
t^{\sigma/2} \tr \partial_{tt} \ve \in L^2(0,T;\Hsd) \quad \sigma > 3-2\gamma.
\end{equation}
We show below that this is a valid assumption provided $\mathcal{A}(\usf_0,f)<\infty$, where
\begin{equation}
\label{A}
\mathcal{A}(\usf_0,f) =  \| \usf_0\|_{\Hs} + \| f\|_{H^2(0,T;\Hsd)}.
\end{equation}
\vspace{-0.45cm}
\begin{theorem}[time regularity for $\gamma \in (0,1)$] 
\label{TH:regularity_in_time}
Assume that $\usf_0 \in \Hs$ and $f \in H^2(0,T;\Hsd)$. Then, for $t \in (0,T\,]$, the solution $\usf$ of \eqref{fractional_heat} satisfies
\begin{equation}
\label{eq:reg_time_betam1}
\| \partial_t \usf(\cdot,t) - \delta^1 \usf(\cdot,t) \|_{\Hsd} \lesssim t^{\gamma-1}\mathcal{A}(\usf_0,f),
\end{equation}
where $\delta^1 \usf(\cdot,t) = t^{-1}\big(\usf(\cdot,t) -\usf(\cdot,0)\big)$. Moreover,
\begin{equation}
\label{eq:reg_time_betam2}
\| t^{\sigma/2} \partial_{tt} \usf \|_{L^2(0,T;\Hsd)} \lesssim \mathcal{A}(\usf_0,f),
\end{equation}
where $\sigma > 3 -2\gamma$. The hidden constant is independent of $t$ but blows up as $\gamma \downarrow0$.
\end{theorem}
\begin{proof} We proceed in three steps and apply the principle of superposition.

\noindent \boxed{1} Case $f \equiv 0$ and $\usf_0 \neq 0$. The solution of \eqref{fractional_heat} is $\usf(x',t) = G_{\gamma}(t) \usf_0 (x')$, which coincides with the solution representation of the alternative formulation of \eqref{fractional_heat} studied in \cite[(2.6)--(2.7)]{M:10}. The regularity results of \cite[Theorem 4.2]{M:10} yield the estimate
$
 \| \partial_{tt} \usf \|_{\Hsd} \lesssim t^{\gamma-2}\| \usf_0 \|_{\Hs}
$
for $t \in (0,T\,]$, whence \eqref{eq:reg_time_betam2} follows.

To derive \eqref{eq:reg_time_betam1} we invoke the fact that $\usf_k$ solves \eqref{uk} with $\usf_k(0) = \usf_{0,k}$ and $f_k \equiv 0$, whence
$\usf_k(t) = E_{\gamma,1}(-\lambda_k^s t^{\gamma})\usf_{0,k}$
according to \eqref{u_exact}. Using \eqref{ML}, $\usf_k(t)$ becomes
$\usf_k(t) = \sum_{m=0}^{\infty} \frac{(-\lambda_k^s t^{\gamma} )^m }{\Gamma(\gamma m + 1)} \usf_{0,k},$
whence
\begin{equation}\label{dt-uk}
\diff_t \usf_k(t) 
= -\usf_{0,k} \lambda_k^s t^{\gamma-1} \sum_{m=0}^{\infty}
\frac{(-\lambda_k^s t^{\gamma})^{m}}{\Gamma(\gamma m + \gamma)}
= - \usf_{0,k} \lambda_k^s t^{\gamma-1} E_{\gamma,\gamma}(-\lambda_k^s t^{\gamma}).
\end{equation}
Likewise, we obtain $\delta^1 \usf_k(t) = - \usf_{0,k} \lambda_k^s t^{\gamma-1}
E_{\gamma,\gamma+1}(-\lambda_k^s t^{\gamma})$.
Therefore, \eqref{eq:reg_time_betam1} follows from \eqref{ML_estimate}.

\noindent \boxed{2} 
Formula \eqref{eq:reg_time_betam1} with $\usf_0 \equiv 0$. We now have 
$\usf(x',t) = \int_{0}^{t}F_{\gamma}(t-r)f(x',r) \diff r$
with $F_{\gamma}$ given by \eqref{barG}. The representation \eqref{u_exact} gives
$
\usf_k(t) = \int_{0}^t r^{\gamma-1}E_{\gamma,\gamma}(-\lambda_k^s r^{\gamma}) f_k(t-r) \diff r,
$
because $\usf_k(0) = 0$. This, combined with \eqref{ML_estimate}, readily implies
\[
\big|\usf_k(t)\big| \le \|f_k\|_{L^\infty(0,T)} \int_0^t r^{\gamma-1}
\diff r \lesssim t^\gamma \|f_k\|_{H^1(0,T)}.
\]
Therefore, \eqref{eq:reg_time_betam1} reduces to deriving suitable bounds for
\begin{equation}\label{1st-derivative}
\diff_t \usf_k(t) = t^{\gamma-1}E_{\gamma,\gamma}(-\lambda_k^s t^{\gamma}) f_k(0) + \int_{0}^t r^{\gamma-1}E_{\gamma,\gamma}(-\lambda_k^s r^{\gamma}) \diff_t f_k(t-r) \diff r.
\end{equation}
The first term yields \eqref{eq:reg_time_betam1} because of \eqref{ML_estimate} and $|f_k(0)| \lesssim \|f_k\|_{H^1(0,T)}$. On the other hand, we use \eqref{ML_estimate} again to bound the second term $\mathfrak{J}_k$ as follows and thus get \eqref{eq:reg_time_betam1}:
\[
\mathfrak{J}_k \le \|\diff_t f_k\|_{L^\infty(0,T)}\int_0^t r^{\gamma-1} \diff r \lesssim t^\gamma \|\diff_t f_k\|_{L^\infty(0,T)}.
\]

\noindent \boxed{3} 
Formula \eqref{eq:reg_time_betam2} with $\usf_0 \equiv 0$. Differentiating \eqref{1st-derivative} once more, we obtain
\begin{equation}\label{2nd-derivative}
\begin{aligned}
\diff_{tt} \usf_k(t) &=
(\gamma-1)t^{\gamma-2}E_{\gamma,\gamma}(-\lambda_k^s t^{\gamma}) f_k(0) 
- \lambda_k^s t^{2(\gamma-1)} E'_{\gamma,\gamma}(-\lambda_k^s t^{\gamma}) f_k(0)
\\ 
& + t^{\gamma-1} E_{\gamma,\gamma}(-\lambda_k^s t^{\gamma}) \diff_t f_k(0)
+ \int_{0}^t r^{\gamma-1}E_{\gamma,\gamma}(-\lambda_k^s
r^{\gamma}) \diff_{tt} f_k(t-r) \diff r,
\end{aligned}
\end{equation}
and employ again \eqref{ML_estimate}. Since, $\sigma>3-2\gamma$ yields $\int_0^T r^{\sigma +2\gamma-4} \diff r<\infty$, the first and third terms lead to  \eqref{eq:reg_time_betam2}. For the second term we resort to the identity $ \gamma z E_{\gamma,\gamma}'(z) = E_{\gamma,\gamma-1}(z)-(\gamma-1)E_{\gamma,\gamma}(z)$ to end up with the same condition on $\sigma$. For the fourth term, we use that $\sigma>1$ and Lemma~\ref{le:continuity} to obtain the
bound 
$
T\|r^{\gamma-1} \star |\diff_{tt} f_k| \|_{L^2(0,T)} \lesssim 
\|\diff_{tt} f_k\|_{L^2(0,T)}
$. This concludes the proof.
\end{proof}

For $\gamma \in (0,1)$ it will be useful, when analyzing fully discrete schemes, to have pointwise estimates for time derivatives of the solution $\ve$. We thus define, for $\mu > 0$,
\begin{equation}
\label{eq:frakB}
\mathcal{B}(\usf_0, f ) := \| \usf_0 \|_{\mathbb{H}^{1+3s}(\Omega)} + \| f|_{t=0} \|_{\mathbb{H}^{1+s}(\Omega)} + \| f \|_{W^1_\infty(0,T;\mathbb{H}^{1-(1-2\mu)s}(\Omega))}.
\end{equation}

\begin{corollary}[pointwise estimate for time derivatives]
\label{CR:Sut}
If $\gamma \in (0,1)$, then
\begin{equation}
\label{Sut}
\mathcal{S}(\ve_t(\cdot,t)) \lesssim t^{\gamma-1} \mathcal{B}(\usf_0, f).
\end{equation}
In addition $I^{1-\gamma}\mathcal{S}(\ve_t) \in L^2(0,T)$ with 
$\|I^{1-\gamma}\mathcal{S}(\ve_t)\|_{L^2(0,T)} \lesssim \mathcal{B}(\usf_0, f)$.
\end{corollary}
\begin{proof}
In view of \eqref{V_yy} and \eqref{D}, as well as $s \in (0,1)$, we see that
$
\mathcal{S}(\ve_t(\cdot,t))^2 \lesssim \sum_{k=0}^\infty \lambda_k^{1+s}
\big|\diff_t \usf_k(t) \big|^2.
$
Since $\diff_t \usf_k(t)$ is the sum of \eqref{dt-uk} and
\eqref{1st-derivative}, we deduce
\[
\big|\diff_t \usf_k(t) \big| \le \big|\usf_{0,k}\big| \lambda_k^s t^{\gamma-1}
+ t^{\gamma-1} \big| f_k(0) \big|
+ \lambda_k^{-s} \log (1+\lambda_k^s T^\gamma) \|\diff_t f_k\|_{L^\infty(0,T)},
\]
where we have used \eqref{log-term}. This readily implies \eqref{Sut}.

We now prove $I^{1-\gamma}\mathcal{S}(\ve_t) \in L^2(0,T)$. For $\gamma \in (\tfrac{1}{2},1)$ this follows from \eqref{Sut} and Corollary~\ref{co:continuity}. If $\gamma \in (0,\tfrac{1}{2}]$, we first note that $t^{\gamma-1} \in L\log L (0,T)$. A generalization of a theorem by Hardy and Littlewood \cite[Theorem 4]{GMJ:5227752} shows that $I^{1-\gamma}:  L \log L(0,T) \rightarrow L^{1/\gamma}(0,T)$ boundedly. Since $1/\gamma \geq 2$, this concludes the proof.
\end{proof}


\section{Time discretization}
\label{sec:time_discretization}

Let $\K \in \mathbb{N}$ denote the number of time steps. We define the uniform time step as $\tau = T/\K > 0$, and set $t_k = k \tau$ for $0 \leq k \leq \K$. We also define $I_k= (t_k,t_{k+1}]$ for $0 \leq k \leq \K-1$. If $\Xcal$ is a normed space with norm $\| \cdot \|_{\Xcal}$, then for $\phi \in C ( [0,T], {\Xcal} )$ we denote $\phi^k = \phi(t_k)$ and 
$\phi^{\tau}= \{ \phi^k\}_{k=0}^{\K}$. Moreover, 
\[
\| \phi^{\tau} \|_{\ell^{\infty}(\Xcal)} = \max_{0 \leq k \leq \K} \| \phi^k\|_{\Xcal},
\qquad \| \phi^{\tau} \|_{\ell^2(\Xcal)}^2 = \sum_{k=1}^\K \tau \| \phi^k\|_{\Xcal}^2.
\]
For a sequence of time-discrete functions $W^{\tau} \subset \Xcal$ we define, for $k=0,\dots,\K-1$,
\begin{equation}
\label{partial}
\delta^1 W^{k+1} = \tau^{-1} (W^{k+1} - W^{k}). 
\end{equation}

\subsection{Time discretization for $\gamma =1$}
\label{sub:discretization_1}
We apply the backward Euler scheme to \eqref{heat_harmonic_extension_weak} for $\gamma=1$: determine $V^\tau = \{ V^k\}_{k=0}^{\K}  \subset \HL(y^{\alpha},\C)$ such that
\begin{equation}
\label{initial_data_cont}
 \tr V^{0} = \usf_0,
\end{equation}
and, for $k=0,\dots, \K-1$, $V^{k+1} \in \HL(y^{\alpha},\C)$ solves
\begin{equation}
\label{first_order}
   \left( \delta^1 \tr  V^{k+1}, \tr W \right)_{L^2(\Omega)}+
    a(V^{k+1},W) = \left\langle f^{k+1}, \tr W   \right\rangle,
 \end{equation}
for all $W \in \HL(y^{\alpha},\C)$, where $f^{k+1} = f(t^{k+1})$. Define $U^\tau = \{ U^k\}_{k=0}^{\K}  \subset \Hs$ with
\begin{equation}
\label{discrete_U}
 U^k:= \tr V^k,
\end{equation}
which is a piecewise constant (in time) approximation of $\usf$, solution to problem \eqref{fractional_heat}. Note that \eqref{initial_data_cont} does not require an extension of $\usf_0$.

\begin{remark}[dynamic boundary condition] \rm
Problem \eqref{initial_data_cont}--\eqref{first_order} is a sequence of elliptic problems with dynamic boundary condition, the discrete counterpart of \eqref{heat_harmonic_extension_weak}. Its analysis is slightly different from the standard theory for parabolic problems.
\end{remark}

\begin{remark}[locality] \rm
The main advantage of scheme \eqref{initial_data_cont}--\eqref{first_order} is its local nature, which mimics that of problem \eqref{heat_harmonic_extension_weak}.
\end{remark}

The stability of this scheme is rather elementary as the following result shows.

\begin{lemma}[unconditional stability for $\gamma=1$]
\label{le:stab_1}
The semi-discrete scheme \eqref{initial_data_cont}--\eqref{first_order} is unconditionally stable, namely
\begin{equation}
\label{stability_1} \| \tr V^{\tau} \|^2_{\ell^{\infty}(L^2(\Omega))} + \| V^{\tau} \|^2_{\ell^2(\HLn(y^{\alpha},\C))}
\lesssim  
\| \usf_0 \|^2_{L^2(\Omega)} + \| f^{\tau} \|^2_{\ell^2(\Hsd)}.
\end{equation}
\end{lemma}
\begin{proof}
Set $W = 2 \tau V^{k+1}$ in \eqref{first_order}. 
Estimate \eqref{Trace_estimate} and Young's inequality yield
\begin{align*}
\| \tr V^{k+1} \|^2_{L^2(\Omega)} - \| \tr V^{k} \|^2_{L^2(\Omega)} + \tau\| V^{k+1} \|^2_{\HLn(y^{\alpha},\C)}
\lesssim & 
\tau \| f^{k+1} \|^2_{\Hsd}.
\end{align*}
Adding this inequality over $k$ yields \eqref{stability_1}.
\end{proof}

\subsection{Time discretization for $\gamma \in (0,1)$}
\label{sub:discretization_beta}

We now discretize the nonlocal operator $\partial_t^{\gamma}$ of order $\gamma \in (0,1)$.
We consider the finite difference scheme proposed in \cite{LLX:11,LinXu:07} but resort to the regularity results of Theorem~\ref{TH:regularity_in_time}. Definition~\ref{caputo}
and the Taylor formula with integral remainder yield, for $0 \leq k \leq \K - 1$,
\begin{equation}
  \begin{aligned}
    \partial_t^{\gamma}\usf(\cdot,t_{k+1}) & = \frac{1}{\Gamma(1-\gamma)}
    \int_0^{t_{k+1}} \frac{\partial_{t}\usf(\cdot,t)}{(t_{k+1} - t)^{\gamma}} \diff t
    \\
    & =
    \frac{1}{\Gamma(1-\gamma)}\sum_{j=0}^{k} \frac{\usf(\cdot,t_{j+1}) - \usf(\cdot,t_{j})}{ \tau }
    \int_{I_j}  \frac{\diff t}{(t_{k+1} - t)^{\gamma}}
    + \resto_{\gamma}^{k+1}(\cdot)
    \\
    & =
    \frac{1}{\Gamma(2-\gamma)}\sum_{j=0}^{k} a_j \frac{\usf(\cdot,t_{k+1-j}) - \usf(\cdot,t_{k-j})}{ \tau^{\gamma} }
    + \resto_{\gamma}^{k+1}(\cdot),
  \label{discretization_fractional}
  \end{aligned}
\end{equation}
where 
\begin{equation}
\label{a_j}
a_j = (j+1)^{1-\gamma} - j^{1-\gamma}, \qquad 
\resto_{\gamma}^{k+1} = \frac{1}{\Gamma(1-\gamma)} \sum_{j=0}^{k} \int_{I_j} \frac{1}{(t_{k+1}-t)^{\gamma}} R(\cdot,t) \diff t
\end{equation}
denotes the remainder and $R$ is defined by
\begin{equation}\label{R}
 R(\cdot,t) = \partial_t \usf(\cdot,t) - \frac{1}{\tau}\big( \usf(\cdot,t_{j+1}) - \usf(\cdot,t_j) \big)
\qquad \forall t \in I_j.
\end{equation}
Notice that from \eqref{a_j} we deduce that $a_j > 0$ for all $j\geq 0$ and
\[
 1 = a_0 > a_1 > a_2 > \dots > a_j, 
\qquad 
\lim_{j \rightarrow \infty} a_j = 0.
\]

\subsubsection{Consistency estimate}
\label{sub:sub:consistency}
We now estimate the residual $\resto_{\gamma}^{\tau}$ by exploiting a cancellation property. We first observe that the function $R$ defined in \eqref{R} has vanishing mean in $I_j$ for all $j \in \{0,\dots,\K-1\}$, whence we can write
\begin{equation}
\label{resto:aux}
\resto_{\gamma}^{k+1} = \frac{1}{\Gamma(1-\gamma)} \sum_{j=0}^{k} \int_{I_j} ( \psi_{\gamma}(t) - \bar{\psi}^j_{\gamma} )R(\cdot,t) \diff t, 
\end{equation}
with $\psi_{\gamma}(t) = (t_{k+1}-t)^{-\gamma}$ and $\bar{\psi}^j_{\gamma} = \fint_{I_j} \psi_{\gamma}(t) \diff t$.
The conclusion of Lemma~\ref{le:continuity} yields
\begin{equation}
\label{eq:restoYoung}
\| \resto_\gamma^\tau \|_{L^2(0,T;\Hsd)} \lesssim \| \psi_{\gamma} - \bar{\psi}^{\tau}_{\gamma}\|_{L^1(0,T)} \| R^\tau \|_{L^2(0,T;\Hsd)},
\end{equation}
which reduces the estimation of the residual to providing suitable bounds for each term on the right hand side of this expression.
We start with $\|R^{\tau}\|_{L^2(0,T;\Hsd)}$.

\begin{lemma}[estimate for $R^{\tau}$] 
\label{le:R}
If $\mathcal{A}(\usf_0,f)< \infty$, then $R^\tau$ defined by \eqref{R} satisfies
\[
  \|R^{\tau}\|_{L^2(0,T;\Hsd)} \lesssim \tau^{1 -\tfrac{\sigma}{2}}\mathcal{A}(\usf_0,f),
\]
for $\sigma > 3-2\gamma$.
\end{lemma}
\begin{proof}
For $1\leq j \leq \K-1$ and $t \in I_j$, \eqref{R} implies
\begin{align*}
\|R(t_{j+1}) \|_{\Hsd}  
\leq \int_{I_j} \| \partial_{tt}\usf(z)\|_{\Hsd} \diff z
\leq  \| z^{-\sigma/2}\|_{L^2(I_j)} \| z^{\sigma/2} \partial_{tt}\usf \|_{L^2(I_j;\Hsd) },
\end{align*}
whence
\begin{align*}
  \tau \sum_{j=2}^{\K} \|R(t_j)\|^2_{\Hsd} & \leq \tau \sum_{j=2}^{\K} \left(\int_{I_j} z^{-\sigma} \diff z\right) \left(\int_{I_j} z^\sigma \|\partial_{tt} \usf \|_{\Hsd}^2 \diff z\right) \\
  & \leq \tau \max_j \left( \int_{I_j} z^{-\sigma} \diff z \right) \| t^{\sigma/2} \partial_{tt} \usf \|_{L^2(\tau,T;\Hsd)}^2\\
  & \leq \tau^{2-\sigma} \| t^{\sigma/2} \partial_{tt} \usf \|_{L^2(\tau,T;\Hsd)}^2
  \lesssim \tau^{2-\sigma} \mathcal{A}(\usf_0,f)^2,
\end{align*}
in view of \eqref{eq:reg_time_betam2}. For the first interval $I_0 = (0,\tau]$, we combine \eqref{eq:reg_time_betam1} with \eqref{R} to get
\begin{align*}
\| R (t_1) \|_{\Hsd} = \left\| \partial_t \usf(t_1) - \delta^1 \usf(t_1)\right\|_{\Hsd} 
\lesssim \tau^{\gamma-1}\mathcal{A}(\usf_0,f). 
\end{align*}
Collecting the preceding estimates we arrive at
\[
 \|R^{\tau}\|^2_{L^2(0,T;\Hsd)}
 =  \sum_{j=1}^{\K} \tau \| R(t_j) \|^2_{\Hsd}
 \lesssim \tau^{2-\sigma}\mathcal{A}(\usf_0,f)^2,
\]
where we have used that $2-\sigma < 2\gamma - 1$. This concludes the proof.
\end{proof}

We now estimate the $L^1$-norm of $\psi_{\gamma} - \bar{\psi}^{\tau}_{\gamma}$.

\begin{lemma}[kernel estimate] The kernel $\psi_{\gamma}=(t_{k+1}-t)^{-\gamma}$ satisfies
\label{le:kernel}
\begin{equation*}
 \| \psi_{\gamma} - \bar{\psi}^{\tau}_{\gamma} \|_{L^1(0,T)} \leq \frac{2-\gamma}{1-\gamma} \tau^{1-\gamma}.
\end{equation*}
\end{lemma}
\begin{proof}
We split the integral over intervals $I_j$. We first consider $0 \leq j < k$:
\begin{align*}
\int_{I_j} | \psi_{\gamma}(t) - \bar{\psi}^j_{\gamma} | \diff t 
& = \frac{1}{\tau} \int_{I_j} \Big| \int_{I_j} ( \psi_{\gamma}(t) - \psi_{\gamma}(r)) \diff r \Big| \diff t
\leq \tau \int_{I_j} | \psi_{\gamma}'(t) |\diff t
\\
& = \tau \gamma  
\int_{I_j} \frac{1}{(t_{k+1}-t)^{\gamma+1}} \diff t = \tau^{1-\gamma}
\left[ \frac{1}{(k-j)^{\gamma}} - \frac{1}{(k-j+1)^{\gamma}} \right],
\end{align*}
If $j = k$ 
set $\bar{\psi}^{k}_{\gamma} = 0$ and 
$
 \int_{I_k} \psi_{\gamma}(t) \diff t = \int_{I_k} (t_{k+1}-t)^{-\gamma} \diff t
 = \frac{\tau^{1-\gamma}}{1-\gamma}.
$
Consequently,
\begin{align*}
 \|  \psi_{\gamma} - \bar{\psi}^{\tau}_{\gamma} \|_{L^1(0,T)} & = \sum_{j=0}^k \int_{t_{j}}^{t_{j+1}} | \psi_{\gamma}(t) - \bar{\psi}^j_{\gamma} | \diff t
\\
& \leq \tau^{1-\gamma} \left( \frac{1}{1-\gamma} + 
\sum_{j=0}^{k-1}\left[ \frac{1}{(k-j)^{\gamma}} - \frac{1}{(k-j+1)^{\gamma}} \right] \right) 
\\
& = \tau^{1-\gamma} \left( \frac{1}{1-\gamma} + 1 -\frac{1}{(k+1)^{\gamma}}  \right)
\leq \frac{2-\gamma}{1-\gamma} \tau^{1-\gamma},
\end{align*}
which concludes the proof.
\end{proof}

We now derive an estimate for $\resto_{\gamma}^{\tau}$, which, although yields lower rates of convergence than \cite[(3.4)]{LLX:11}, takes into account the correct behavior of the solution and the singularity of its derivatives as $t\downarrow0$.

\begin{proposition}[consistency]
\label{pro:consistency}
The fractional residual $\resto_{\gamma}^{\tau} = \{ \resto_{\gamma}^{k} \}_{k=0}^{\K}$ satisfies
\begin{equation}
\label{consistency}
\| \resto_{\gamma}^{\tau} \|_{L^2(0,T;\Hsd)} \lesssim \tau^{\theta} \mathcal{A}(\usf_0,f) \quad 0< \theta < \frac{1}{2}.
\end{equation}
The hidden constant is independent of the data and $\tau$ but blows up as $\theta \uparrow \frac{1}{2}$.
\end{proposition}
\begin{proof}
The assertion follows from \eqref{eq:restoYoung} and Lemmas \ref{le:R} and \ref{le:kernel}.
\end{proof}

\subsubsection{Abstract stability and energy estimates}
\label{sub:sub:energy}

To fix the ideas concerning the application of the discretization \eqref{discretization_fractional}, we present an approach within a general Hilbert space setting. Given a Gelfand triple $\Vcal \subset \Hcal \equiv \Hcal' \subset \Vcal'$, let $\mathfrak{F}: \Vcal \rightarrow \Vcal'$ be a linear, continuous and coercive operator. If $(\cdot,\cdot)_{\Hcal}$ is the inner product in $\Hcal$, set
\[
 \| U \|_{\Hcal} = (U,U)_{\Hcal}^{1/2}, \qquad \| U \|_{\Vcal} = \langle \mathfrak{F} U,U\rangle^{1/2},
\]
where $\langle \cdot,\cdot \rangle$ denotes the duality pairing
between $\Vcal$ and $\Vcal'$.
Given $f \in L^2(0,T;\Vcal')$ and $\usf_0 \in \Hcal$, 
we study a time discretization scheme for the fractional evolution problem
\begin{equation}
\label{fractional_A}
\partial^{\gamma}_t \usf + \mathfrak{F} \usf = f, \qquad  \usf(0) = \usf_0.
\end{equation}

If $\gamma\in(0,1)$ and $\phi^\tau \subset \Hcal$, we define, according to \eqref{discretization_fractional},  the discrete fractional derivative, for $k=0, \ldots, \K-1$ by
\begin{equation}
\label{L_discrete}
\Gamma(2-\gamma)\delta^\gamma \phi^{k+1} := \sum_{j=0}^{k} \frac{  a_j }{ \tau^{\gamma-1} } \delta^1\phi^{k+1-j}
= \frac{\phi^{k+1}}{\tau^\gamma} - \sum_{j=0}^{k-1} \frac{ a_j - a_{j+1} }{\tau^\gamma} \phi^{k-j} - \frac{a_k}{\tau^\gamma} \phi^0
\end{equation}
where the second equality holds because $a_0 = 1$ and the sum for $k=0$ is defined to be zero.
The implicit semi-discrete scheme to solve \eqref{fractional_A} reads: Let $U^{0} = \usf_0$ and,
for $k=0,\dots, \K-1$, let $U^{k+1} \in \Vcal$ solve
\begin{equation}
\label{discrete_beta_abs}
(\delta^\gamma U^{k+1},W)_{\Hcal} + \langle \mathfrak{F} U^{k+1},W\rangle= \langle f^{k+1},W \rangle,\
\quad \forall W \in \Vcal.
\end{equation}

We have the following stability result.
\begin{theorem}[unconditional stability for $\gamma \in (0,1)$]
\label{TH:stabbeta}
The implicit semi-discrete scheme \eqref{discrete_beta_abs} is unconditionally stable and satisfies
\begin{equation}
\label{stab_beta_re}
I^{1-\gamma} \| U^{\tau}\|_\Hcal^2(T)
+ \| U^{\tau} \|^2_{\ell^2(\Vcal)} \leq 
  I^{1-\gamma} \| U^0\|_\Hcal^2(T) +  \| f^{\tau} \|^2_{\ell^2(\Vcal')}.
\end{equation}
\end{theorem}
\begin{proof}
Denote $\kappa = \Gamma(2-\gamma)\tau^{\gamma}$ and set $W = 2 \kappa U^{k+1}$ in \eqref{discrete_beta_abs}.
We obtain
\begin{align*}
 2\|U^{k+1} \|^2_{\Hcal}
& + 2 \kappa \| U^{k+1} \|^2_{\Vcal}
\\
& = 2\sum_{j=0}^{k-1} ( a_j - a_{j+1} )(U^{k-j},U^{k+1})_{\Hcal}
+ 2a_k (U^0,U^{k+1})_{\Hcal} + 2 \kappa \langle f^{k+1},U^{k+1} \rangle,
\end{align*}
for $0 \leq k \leq \K -1$ provided the sum vanishes for $k=0$. Using the Cauchy-Schwarz inequality, the fact that $a_j - a_{j+1}>0$, and the telescopic property of the sum 
$
\sum_{j=0}^{k-1} (a_j - a_{j+1}) = 1- a_k$, we obtain for $0 \leq k \leq \K -1
$ 
\begin{align*}
\nonumber
\left(2 -(1-a_k) -a_k \right)\|U^{k+1} \|^2_{\Hcal} & + \kappa \| U^{k+1} \|^2_{\Vcal}
\\
& \leq \sum_{j=0}^{k-1} ( a_j - a_{j+1} ) \|U^{k-j} \|^2_{\Hcal} + a_k \| U^0\|^2_{\Hcal} + \kappa \| f^{k+1}\|^2_{\Vcal'}.
\end{align*}
A simple manipulation of the left-hand side of this inequality yields
\[
\sum_{j=0}^{k} a_j \|U^{k+1-j} \|^2_{\Hcal} + \kappa   \| U^{k+1} \|^2_{\Vcal} 
\leq \sum_{j=0}^{k-1} a_j \|U^{k-j} \|^2_{\Hcal} + a_k\| U^0\|^2_{\Hcal} + \kappa \| f^{k+1}\|^2_{\Vcal'},
\]
where the sum on the right-hand side vanishes for $k=0$. Adding over $k$ we get
\[ 
\sum_{j=0}^{\K-1} a_j \|U^{\K-j} \|^2_{\Hcal}
+ \kappa \sum_{k=1}^{\K}\| U^{k} \|^2_{\Vcal} \leq 
  \left( \sum_{k=0}^{\K-1}a_k \right)
  \| U^0 \|_{\Hcal}^2 +  \kappa \sum_{k=1}^{\K}\| f^{k} \|^2_{\Vcal'}.
\]
Since $I^{1-\gamma}1(T) = \frac{\tau^{1-\gamma}}{\Gamma(2-\gamma)} \sum_{k=0}^{\K-1} a_k$, multiplying
this inequality by $\frac{\tau^{1-\gamma}}{\Gamma(2-\gamma)}$, we obtain
\begin{equation}
\label{stab_beta}
\frac{\tau^{1-\gamma}}{\Gamma(2-\gamma)} \sum_{j=0}^{\K-1} a_j \|U^{\K-j} \|^2_{\Hcal} + \| U^{\tau} \|^2_{\ell^2(\Vcal)} \leq I^{1-\gamma} \| U^0 \|_{\Hcal}^2(T) +  \| f^{\tau} \|^2_{\ell^2(\Vcal')}.
\end{equation}
Now, changing the summation index and using the definition \eqref{a_j}, we obtain
\begin{align*}
\sum_{j=0}^{\K-1} a_j \|U^{\K-j} \|^2_{\Hcal} 
& =  \frac{1}{\tau^{1-\gamma}}
\sum_{l=1}^{\K} \left( (T - t_{l-1} )^{1-\gamma} - (T - t_{l} )^{1-\gamma}\right) \|U^l \|^2_{\Hcal} 
\\
& = \frac{1-\gamma}{\tau^{1-\gamma}} \sum_{l=1}^{\K} \int_{t_{l-1}}^{t_l} \frac{\|U^{\tau}(r) \|^2_{\Hcal} }{(T-r)^{\gamma}}
\diff r,
\end{align*}
whence
$
\frac{\tau^{1-\gamma}}{\Gamma(2-\gamma)}\sum_{j=0}^{\K-1} a_j \|U^{\K-j} \|^2_{\Hcal} 
= I^{1-\gamma}\|U^\tau \|^2_{\Hcal}(T),
$
which together with \eqref{stab_beta} yields the desired estimate \eqref{stab_beta_re}.
\end{proof}

Deducing an energy estimate for problem \eqref{fractional_A} is nontrivial due to the nonlocality of the fractional time derivative. The main technical difficulty lies on the fact that a key ingredient in deriving such a result is 
an integration by parts formula, which for a function $\usf$ not vanishing at $t=0$ and $t=T$ involves boundary terms that need to be estimated; for a step in this direction see \cite{ER:06,LiXu:09}. In this sense, the discrete energy estimate \eqref{stab_beta_re} has an important consequence at the continuous level. 

\begin{corollary}[fractional energy estimate for $\usf$]
\label{CO:energybeta}
Let $\gamma \in (0,1)$. Then, 
\begin{equation}
\label{stab_beta_re2}
I^{1-\gamma} \|\usf\|_\Hcal^2(T)
+ \|\usf \|^2_{L^2(0,T;\Vcal)} \leq 
I^{1-\gamma} \| \usf_0 \|_\Hcal^2(T) +  \| f \|^2_{L^2(0,T;\Vcal')}.
\end{equation}
\end{corollary}
\begin{proof}
Given that the estimate \eqref{stab_beta_re} is uniform in $\tau$,
and $\| \resto_{\gamma}^{k+1} \|_{L^2(0,T;\Vcal')} \lesssim \tau^{\theta}$ with $0 < \theta < \frac{1}{2}$, we easily derive \eqref{stab_beta_re2} by taking $\tau \downarrow 0$ in \eqref{stab_beta_re}.
\end{proof}

\begin{remark}[limiting case] 
\label{RE:militing_case} \rm
Given $g \in L^p(0,T)$,  we have $I^\sigma g \rightarrow g$ 
in $L^p(0,T)$ as $\sigma \downarrow 0$; see \cite[Theorem 2.6]{Samko}. This implies that, taking the limit as $\gamma \uparrow 1$
in \eqref{stab_beta_re2}, we recover the well known stability result for a parabolic equation, \ie
\begin{equation}
\label{stab_beta_1}
\|\usf\|_{L^\infty(0,T;\Hcal)}^2
+ \|\usf \|^2_{L^2(0,T;\Vcal)} \leq 
\| \usf_0 \|_\Hcal^2 +  \| f \|^2_{L^2(0,T;\Vcal')}.
\end{equation}
This allows us to unify the estimate of Corollary~\ref{CO:energybeta} for all $\gamma \in (0,1]$.
\end{remark}

\subsection{Discrete stability}
\label{sub:Stab}

We now apply the ideas developed in \S \ref{sub:discretization_1} and \S \ref{sub:discretization_beta} to problem \eqref{fractional_heat}, \ie we consider $\mathfrak{F} = \calLs$. As it was discussed in \S\ref{sub:CaffarelliSilvestre}, we realize the nonlocal spatial operator $\calLs$ with the Caffarelli-Silvestre extension and look for solutions of the extended problem \eqref{heat_harmonic_extension_weak}. In view of \eqref{first_order} and \eqref{discrete_beta_abs}, we propose the following \emph{semi-discrete} numerical scheme to approximate problem \eqref{heat_harmonic_extension_weak} for $\gamma \in (0,1]$:

Set $\tr V^{0} = \usf_0$. For $k=0,\dots, \K-1$ find $V^{k+1} \in \HL(y^{\alpha},\C)$, solution of
\begin{equation}
\label{discrete_beta}
(  \delta^\gamma \tr V^{k+1},\tr W)_{L^2(\Omega)} 
+ 
a(V^{k+1},W) = \langle f^{k+1}, \tr W \rangle,
\end{equation}
for all  $W \in\HL(y^{\alpha},\C)$, where $a$ is the bilinear form defined in \eqref{a}, and $\delta^\gamma$ is defined by \eqref{L_discrete} for $\gamma \in (0,1)$ and \eqref{partial} for $\gamma = 1$. We have the following stability result. 

\begin{corollary}[unconditional stability for $0 < \gamma \leq 1$]
\label{co:stabbeta}
The semi-discrete scheme \eqref{discrete_beta} is unconditionally stable and satisfies
\begin{align}
\nonumber
I^{1-\gamma} \| \tr V^{\tau}\|_{L^2(\Omega)}^2(T) & + \| V^{\tau} \|^2_{\ell^2(\HLn(y^{\alpha},\C))} 
\\ 
\label{stab_beta_re_2}
& \lesssim 
I^{1-\gamma} \| \usf_0 \|_{L^2(\Omega)}^2(T) +  \| f^{\tau} \|^2_{\ell^2(\Hsd)}.
\end{align}
\end{corollary}
\begin{proof}
Set $\mathcal{V} = \Hs$ and $\mathcal{H}=L^2(\Omega)$, and apply
Theorem \ref{TH:stabbeta} for $\gamma \in (0,1)$ and Lemma \ref{le:stab_1} for $\gamma = 1$. 
\end{proof}

\subsection{Error Estimates}
\label{sub:ErrTimedisc}
We present semi-discrete error estimates for \eqref{discrete_beta}.
\begin{theorem}[error estimates for semi-discrete schemes]
\label{thm:semidicserr}
Denote by $\ve$ and $V^\tau$ the solutions to \eqref{heat_harmonic_extension_weak} and \eqref{discrete_beta}, respectively. If $\gamma \in (0,1)$ and $\mathcal{A}(\usf_0,f) < \infty$,
then
\begin{equation}
\label{eq:semi-discreteg}
\left[ I^{1-\gamma}\|\tr (\ve^\tau - V^\tau) \|_{L^2(\Omega)}^2(T) \right]^{\tfrac{1}{2}} 
+ \|  \ve^\tau - V^\tau \|_{\ell^{2}(\HLn(y^\alpha,\C))}
\lesssim \tau^{\theta} \mathcal{A}(\usf_0,f), 
\end{equation}
where $0 < \theta < \frac{1}{2}$ and the hidden constant is independent of the 
data and $\tau$ but blows up for $\theta \uparrow \frac{1}{2}$.
If, on the other hand $\gamma = 1$, then we have
\begin{align}
\nonumber
\|\tr (\ve^\tau - V^\tau) \|_{\ell^\infty(L^2(\Omega))} & +
\|  \ve^\tau - V^\tau \|_{\ell^{2}(\HLn(y^\alpha,\C))}
\\
\label{eq:semi-discrete1}
& \lesssim \tau^{\frac{1}{2}} \left( \| \usf_0\|_{\Hs} + \|f\|_{L^2(0,T;L^2(\Omega))} \right),
\end{align}
or
\begin{align}
\nonumber
\|\tr (\ve^\tau - V^\tau) \|_{\ell^\infty(L^2(\Omega))} & + \|  \ve^\tau - V^\tau \|_{\ell^{2}(\HLn(y^\alpha,\C))}
\\
\label{eq:semi-discrete1D}
&\lesssim \tau\left(  \| \usf_0 \|_{\mathbb{H}^{2s}(\Omega)} + \| f \|_{BV(0,T;L^2(\Omega))} \right),
\end{align}
where, again, the hidden constant is independent of the data and $\tau$.
\end{theorem}
\begin{proof}
Combining \eqref{heat_harmonic_extension_weak} with \eqref{discretization_fractional} and \eqref{L_discrete}, and 
subtracting \eqref{discrete_beta}, the equation for the error
$E^k:= \ve^k - V^k$ reads
\[
(  \delta^\gamma \tr E^{k+1},\tr W)_{L^2(\Omega)} 
+ 
a(E^{k+1},W) = - \langle \resto^{k+1}_{\gamma}, \tr W \rangle.
\]
For $\gamma \in (0,1)$, we apply \eqref{stab_beta_re_2} in conjunction with \eqref{consistency} to derive \eqref{eq:semi-discreteg}. The estimates \eqref{eq:semi-discrete1} and \eqref{eq:semi-discrete1D} follow from \cite[Theorem 3.16]{NSV:00} and \cite[Theorem 3.20]{NSV:00} or \cite{Rulla}, respectively.
\end{proof}

\begin{remark}[error estimates for $\gamma=1$] \rm
\label{rk:gamma1}
Paper \cite{Rulla} shows that under the assumptions $\usf_0 \in \Hs$ and $f \in L^2(0,T;L^2(\Omega))$, the error estimate \eqref{eq:semi-discrete1} is sharp. 
\end{remark}

\section{Space Discretization}
\label{sec:space_discretization}
We now study space discretization of \eqref{heat_harmonic_extension_weak}.
\subsection{Truncation}
\label{sub:truncation}
A first step towards the discretization is to truncate the domain $\C$. Since $\ve(t)$ decays exponentially in the extended direction $y$, for a.e.~$t \in (0,T)$, we truncate $\C$ to $\C_{\Y}=\Omega \times (0,\Y)$ for a suitable $\Y$ and seek solutions in this bounded domain; see \cite[\S3]{NOS}. The next result is an adaptation of \cite[Proposition 3.1]{NOS} and shows the exponential decay of $\ve$. To write such a result, we first define for $\gamma \in (0,1]$
\begin{equation}\label{Lambda}
 \Lambda_{\gamma}^2(\usf_0,f):= I^{1-\gamma}  \| \usf_0 \|^2_{L^2(\Omega)}(T)
+ \|f\|^2_{L^2(0,T;\Hsd)},
\end{equation}
where $I^0$ is the identity according to Remark~\ref{RE:militing_case} (case $\gamma=1$).

\begin{proposition}[exponential decay]
\label{pro:energyYinf}
Given $\gamma \in (0,1]$, and $s \in (0,1)$, we have
\begin{equation}
\label{energyYinf}
\|\nabla \ve\|_{L^2\left( 0,T; L^2(y^{\alpha},\Omega \times (\Y,\infty)) \right) } \lesssim e^{-\sqrt{\lambda_1} \Y/2}
\Lambda_{\gamma}(\usf_0,f),
\end{equation}
where $\Y > 1$ and $\ve$ denotes the solution to \eqref{heat_harmonic_extension_weak}.
\end{proposition}
\begin{proof}
Recall from \eqref{exactforms} that $\ve(x,t) =  \sum_{k} \usf_k(t) \varphi_k(x') \psi_k(y)$ solves 
\eqref{heat_harmonic_extension_weak}. Since $\{ \varphi_k\}_{k\in \mathbb N}$ is an orthonormal basis 
of $L^2(\Omega)$ that satisfies \eqref{eigenvalue_problem_L} we have
\begin{align*}
\int_0^T \int_{ \C \setminus \C_{\Y} } y^{\alpha} |\nabla \ve(x,t)|^2 \diff x \diff t
& \lesssim 
\int_0^T \sum_{k=1}^{\infty} \usf_k(t)^2 \int_{\Y}^{\infty} y^{\alpha} \left( \lambda_k \psi_k(y)^2 + \psi_k'(y)^2 \right)
\diff y \diff t
\\
& = \sum_{k=1}^{\infty} \left| \Y^{\alpha} \psi_k(\Y) \psi_k'(\Y) \right| \int_0^T \usf_k(t)^2 \diff t.
\end{align*}
where we used \eqref{besselenergy}. Since 
$ | \Y^{\alpha} \psi_k(\Y) \psi_k'(\Y) | \lesssim \lambda_k^s e^{-\sqrt{\lambda_k}\Y}$ \cite[(2.32)]{NOS}, we deduce
\[
\int_0^T \int_{ \C \setminus \C_{\Y} } y^{\alpha} |\nabla \ve(x,t)|^2 \diff x \diff t
\lesssim e^{-\sqrt{\lambda_1}\Y} \| \usf \|_{L^2(0,T;\Hs)}^2.
\]
Finally, by setting $\Vcal = \Hs$ and $\Hcal = L^2(\Omega)$, the estimate \eqref{energyYinf} follows from either 
\eqref{stab_beta_re2} for $\gamma \in (0,1)$ or \eqref{stab_beta_1} for $\gamma=1$.
\end{proof}

As a consequence of Proposition \ref{pro:energyYinf}, we can consider the truncated problem
\begin{equation}
\label{heat_alpha_extension_truncated}
\begin{dcases}
  -\DIV \left( y^{\alpha} \mathbf{A} \nabla v \right) + y^{\alpha} cv = 0 \ \textrm{in } \C_{\Y}\times(0,T),  
   \quad v = 0 \ \textrm{on } (\partial_L \C_{\Y} \cup \Omega_{\Y}) \times (0,T)
    \\
   d_s \partial_t^{\gamma} \tr v + \partial_{\nu}^{\alpha} v = d_s f \ \textrm{on } (\Omega \times \{ 0\})\times(0,T), \quad
   v = \usf_0 \ \textrm{on }\Omega \times \{ 0\},~t=0,
\end{dcases}
\end{equation}
where $\Omega_{\Y} = \Omega \times \{ \Y \}$ and $\Y \geq 1$ is sufficiently large. We now define
\begin{gather*}
  \HL(y^{\alpha},\C_\Y)  = \left\{ w \in H^1(y^{\alpha},\C_\Y): w = 0 \text{ on }
    \partial_L \C_\Y \cup \Omega_{\Y} \right\},\\
\mathbb{V}_{\Y}  = \big\{ w \in L^2(0,T;\HL(y^{\alpha},\C_{\Y}) ): \partial_t^{\gamma} 
\tr w  \in L^2(0,T;\Hsd)\big\}.
\end{gather*}
Problem \eqref{heat_alpha_extension_truncated} is then understood as follows: seek $v \in \mathbb{V}_{\Y}$ such that,
for a.e.~$t \in (0,T)$,
\begin{equation}
\label{heat_harmonic_extension_weak_truncated}
\langle \partial_t^{\gamma} \tr  v, \tr \phi \rangle + 
a_\Y(v,\phi) = \langle f, \tr \phi \rangle,
\end{equation}
for all $\phi \in \HL(y^{\alpha},\C_{\Y})$ and $\tr v(0) = \usf_0$. Here
\begin{equation}
\label{a_Y}
  a_\Y(w,\phi) =   \frac{1}{d_s}\int_{\C_\Y} {y^{\alpha}\mathbf{A}(x)} \nabla w \cdot \nabla \phi
 + y^{\alpha} c(x') w \phi .
\end{equation}

\begin{remark}[initial datum]
\label{rem:initialdata} \rm
We define $v(0) \in \HL(y^{\alpha},\C_{\Y})$ as the solution
to \eqref{heat_alpha_extension_truncated} with the Neumann condition replaced
by $\tr v = \usf_{0}$. The following estimate holds:
$
 \| v(0) \|_{\HLn(y^{\alpha},\C_{\Y})} \lesssim \| \usf_0 \|_{\Hs} 
$
\cite[Remark 3.4]{NOS}.
\end{remark}

\begin{lemma}[exponential convergence]
\label{LE:exp_convergence}
For every $\gamma \in (0,1]$ and $\Y \geq 1$, we have
\begin{equation}
\label{exp_convergence}
I^{1-\gamma} \| \tr(\ve-v)\|^2_{L^2(\Omega) }(T)  + 
 \| \nabla(\ve-v)\|^2_{L^2(0,T;L^2(y^\alpha, \C_{\Y}) )} 
\lesssim e^{-\sqrt{\lambda_1} \Y} \Lambda_{\gamma}^2(\usf_0,f),
\end{equation}
where $\ve$ solves \eqref{heat_harmonic_extension_weak}, $v$ solves 
\eqref{heat_alpha_extension_truncated} and
$\Lambda_{\gamma}(\usf,f)$ is defined in \eqref{Lambda}.
 \end{lemma}
\begin{proof}
Let 
$
w(x,t) := \ve(x',y,t) - \ve(x',\Y,t) \in \HL(y^{\alpha},\C_{\Y}) 
$
be a modification of $\ve$ with vanishing trace at $y = \Y$. 
We observe that $w$ satisfies
\begin{multline*}
\langle \tr  \partial_t^{\gamma} w, \tr \phi \rangle + 
a_\Y(w,\phi) = \langle f, \tr \phi \rangle 
- \langle \tr  \partial_t^{\gamma} \ve(\cdot,\Y,\cdot), \tr \phi \rangle
-a_\Y(\ve(\cdot,\Y,\cdot),\phi)
\end{multline*}
for all $\phi \in \HL(y^{\alpha},\C_{\Y})$. Therefore, the error $e:= v - w$ satisfies
\[
 \langle \tr  \partial_t^{\gamma} e, \tr \phi \rangle + 
a_\Y(e,\phi) = 
a_\Y(\ve(\cdot,\Y,\cdot),\phi)  
+ \langle \tr  \partial_t^{\gamma} \ve(\cdot,\Y,\cdot), \tr \phi \rangle.
\]
Setting $\Vcal = \Hs$ and $\Hcal = L^2(\Omega)$, the assertion is a consequence of Corollary~\ref{CO:energybeta} for $\gamma < 1$ and Remark~\ref{RE:militing_case} for $\gamma = 1$, provided we can estimate the right-hand side of the previous expression and $e(\cdot,0) = \ve(\cdot,\Y,0)$. We estimate the three terms in question separately using Proposition \ref{pro:energyYinf} and the representation formula \eqref{exactforms}.

We note first that
$
| a_\Y(\ve(\cdot,\Y,\cdot),\phi) | \lesssim \|  \ve(\cdot,\Y,\cdot) \|_{\HLn(y^{\alpha},\C_{\Y})}
\|  \phi \|_{\HLn(y^{\alpha},\C_{\Y})}
$
and
\[
 \|  \nabla \ve (\cdot,\Y,\cdot) \|^2_{L^2(y^{\alpha},\C_{\Y})} = \frac{1}{\alpha + 1}
\sum_{k=1}^\infty \lambda_k \usf_k^2(t) \Y^{1+\alpha} \psi_k^2(\Y).
\]
Now, since $| \psi_k(y) | \lesssim (\sqrt{\lambda_k}y)^s e^{-\sqrt{\lambda_k}y}$
for $y \geq 1$, we easily see that
\begin{align*}
\|  \nabla \ve (\cdot,\Y,\cdot) \|^2_{L^2(0,T;L^2(y^{\alpha},\C_{\Y}))}  
& \lesssim \Y^{2(1-s)} \sum_{k=1}^\infty \lambda_k \int_0^T \usf_k^2(t)\diff t (\sqrt{\lambda_k}\Y)^{2s} e^{-2\sqrt{\lambda_k}\Y} \\
& \lesssim  e^{-\sqrt{\lambda_1}\Y} \sum_{k=1}^\infty \!  \lambda_k^{s} \int_0^T \!\! \usf_k^2(t) \diff t 
 =  e^{-\sqrt{\lambda_1}\Y} \| \usf \|^2_{L^2(0,T;\mathbb{H}^{s}(\Omega))}.
\end{align*}
For the second term, we have $\partial_t^{\gamma} \ve (\cdot,\Y,t) =
 \sum_{k=1}^\infty \partial_t^{\gamma} \usf_k(t) \varphi_k \psi_k(\Y)$, whence
\begin{align*}
\| \partial_t^{\gamma} \ve (\cdot,\Y,t) \|_{\Hsd}^{2} 
= \sum_{k=1}^\infty | \partial_t^{\gamma} \usf_k(t)|^2 \lambda_k^{-s} |\psi_k(\Y)|^2
\lesssim e^{- \sqrt{\lambda_1}\Y} \sum_{k=1}^\infty | \partial_t^{\gamma} \usf_k(t)|^2 \lambda_k^{-s}.
\end{align*}
On the other hand, in light of \eqref{uk}, we deduce
\[
 \sum_{k=1}^\infty | \partial_t^{\gamma} \usf_k(t)|^2 \lambda_k^{-s} \lesssim 
 \sum_{k=1}^\infty  \left( \usf^2_k(t) \lambda_k^{s} + f_k^2(t) \lambda_{k}^{-s} \right) = \| \usf(t)\|_{\Hs}^2
 + \| f(t)\|_{\Hsd}^2.
\]
Finally,
$
\|   \ve (\cdot,\Y,0) \|^2_{L^2(\Omega)} = \sum_{k=1}^\infty \usf_k^2(0) \psi_k^2(\Y) \lesssim e^{-\sqrt{\lambda_1}\Y} \| \usf_0 \|_{L^2(\Omega)}^{2}.
$
Collecting the previous estimates and invoking the stability bounds \eqref{stab_beta_re2}
and \eqref{stab_beta_1} for $\usf$, we deduce
\begin{equation}
\label{e_exp}
 I^{1-\gamma} \| \tr e\|^2_{L^2(\Omega) }(T)  + 
 \| \nabla e\|^2_{L^2(0,T;L^2(y^{\alpha},\C_{\Y}) )} 
 \lesssim e^{-\sqrt{\lambda_1}\Y} 
 \Lambda_{\gamma}^2(\usf_0,f).
\end{equation}
Moreover, we have
\begin{equation*}
I^{1-\gamma} \| \tr \ve(\cdot,\Y,\cdot,)\|^2_{L^2(\Omega) }(T)  \! + 
\| \ve(\cdot,\Y,\cdot)\|^2_{L^2(0,T;\HLn(y^\alpha,\C_{\Y}) )}  
\lesssim e^{-\sqrt{\lambda_1}\Y} \Lambda_{\gamma}^2(\usf_0,f),
\end{equation*}
which together with \eqref{e_exp} implies the desired estimate \eqref{exp_convergence}.
\end{proof}

As in \S\ref{sec:time_discretization}, we consider a semi-discrete approximation
of \eqref{heat_harmonic_extension_weak_truncated}. Given the initialization
$\tr \Vt^{0} = \usf_0$, for $k=0,\dots, \K-1$, $\Vt^{k+1} \in \HL(y^{\alpha},\C)$ solve
\begin{equation}
\label{discrete_beta_truncated}
(  \delta^\gamma \tr \Vt^{k+1},\tr W)_{L^2(\Omega)} +
a_\Y(\Vt^{k+1},W) = \langle f^{k+1},W \rangle, \quad
\forall W \in\HL(y^\alpha,\C_{\Y}).
\end{equation}
Its stability follows from Lemma~\ref{le:stab_1} ($\gamma = 1$) and 
Theorem~\ref{TH:stabbeta} ($\gamma <1$). We can also prove estimates like those of Theorem~\ref{thm:semidicserr}.
We conclude with the following remark.

\begin{remark}[regularity of $v$ vs.~$\ve$] \rm
\label{rem:reg_of_v}
In \S\ref{sec:fully_scheme} we will approximate $v$, solution to problem \eqref{heat_alpha_extension_truncated}, so it is essential to elucidate its regularity. Separation of variables yields $v(x',y,t) = \sum_{k} v_k(t) \varphi_k(x') \chi_k(y)$, where $\varphi_k$ solves \eqref{eigenvalue_problem_L} and $\chi_k$ solves 
\begin{equation}
\label{chik}
\chi_k'' + \alpha y^{-1} \chi_k' - \lambda_k \chi_k = 0, \quad \chi_k(0) = 1, \quad \chi_k(\Y)= 0.
\end{equation}
Let $I_s$ and $K_s$ be the modified Bessel functions of first and second kind \cite[\S 9.6]{Abra}, then 
\[
  \chi_k(y) = \Big( \sqrt{\lambda_k} y \Big)^s \Big( a_{k,s} K_s( \sqrt{\lambda_k} y ) + b_{k,s} I_s( \sqrt{\lambda_k} y ) \Big) =: \chi_{k,1}(y) + \chi_{k,2}(y).
\] 
To understand the behavior of $\chi_k$, we present the following properties of $I_s$ \cite{Abra}: 
\begin{enumerate}[(a)]
 
 \item\label{1} For $\nu \in \mathbb{R}$, $\lim_{z \downarrow0} 2^{\nu} \Gamma(\nu+1) z^{-\nu}I_\nu(z) = 1$ (see \cite[(9.6.7)]{Abra}).
 
 \item\label{2} For $\nu \in \mathbb{R}$ and $k \in \mathbb{N}$, $ (z^{-1} \diff_z)^k(z^{\nu} I_\nu (z) ) = z^{\nu-k}I_{\nu-k}(z)$ (see \cite[(9.6.28)]{Abra}).  
 
 \item\label{3} For $z \geq 1$, the function $I_{\nu}(z)$ increases as $e^{z}/{\sqrt{2 \pi z}}$ (see \cite[(9.7.1)]{Abra}).
\end{enumerate}
Property \eqref{1} yields $I_{s}(0) = 0$, which together with $\chi_k(0) = 1$ imply $\chi_{k,1} \equiv \psi_k$ and $a_{k,s} = c_s = 2^{1-s}/\Gamma(s)$, where $\psi_k$ solves \eqref{psik}. Since $\chi_k(\Y) = 0$ we obtain
\[
  b_{k,s} = - c_s K_s(\sqrt{\lambda_k} \Y)I_s(\sqrt{\lambda_k} \Y)^{-1}, 
\]
and thus $\chi_{k,2}$. From \eqref{3} and \cite[(v)]{NOS} we have that $\{ b_{k,s}\}_{k \in \mathbb{N}}$ converges exponentially to $0$ as $k \uparrow \infty$, and in particular it is bounded. Now \eqref{asym-psi_k-p}, \eqref{1} and \eqref{2}, with $k=1$, imply that $\lim_{y \downarrow0} y^{\alpha} \chi_k'(y) = \lambda_k^s(e_{k,s}-d_s) $, where $d_s=2^{\alpha}\Gamma(s)/\Gamma(1-s)$ and $e_{k,s} = 2^{1-s} b_{k,s}/\Gamma(s)$. 
This, together with the fact that $\chi_k(y)$ solves \eqref{chik}, yield $\int_{0}^{\Y} y^{\alpha} (\lambda_k \chi_k(y)^2 + \chi_k'(y)^2) \diff y \lesssim \lambda_k^s(e_{k,s}-d_s)$. With these properties, and the fact that $b_{k,s}$ converges exponentially to $0$ as $k \uparrow \infty$, we arrive at
\begin{equation}
\label{key}
 \int_{0+} y^{\beta}|\chi_k''(y)|^2\diff y \lesssim \lambda_k^{3/2-\beta/2} \leq \lambda_k^{2s}, \qquad 
 \mathcal{D}(v(\cdot,t)) \lesssim \sum_{k=1}^{\infty} \lambda_k^{1+s} |v_k(t)|^2,
\end{equation}
where $\mathcal{D}$ is defined right before \eqref{D}. From \eqref{eq:aux_key} $v_k$ solves:
\[
\partial^{\gamma}_t v_k(t) + \lambda_k^s(1+ \tfrac{e_{k,s}}{d_s}) v_k(t) = f_k(t), \ t>0, \qquad v_k(0) = \usf_{0,k},
\]
Estimates \eqref{key} and the exponential convergence of $\{ e_{k,s} \}_{k\in \mathbb{N}}$ allow us to conclude that the regularity of Theorems \ref{TH:regularity} and \ref{TH:regularity_in_time} also holds for $v$.
\end{remark}

\subsection{Finite element methods}
\label{sec:sub:space_discretization}

We follow \cite[\S4]{NOS} and let $\partial\Omega$ be polyhedral. Let $\T_{\Omega} = \{K\}$ be a conforming mesh of $\Omega$ into cells $K$ (simplices or $n$-rectangles):
\[
  \bar\Omega = \bigcup_{K \in \T_{\Omega}} K, \qquad
  |\Omega| = \sum_{K \in \T_{\Omega}} |K|.
\]
Let $\Tr_\Omega$ be a collection of conforming shape-regular refinements $\T_{\Omega}$ of an original mesh $\T_{\Omega}^0$ \cite{CiarletBook}. If $\T_\Omega \in \Tr_\Omega$ we define $h_{\T_{\Omega}}= \max_{K \in \T_\Omega}h_K$. 

We define $\T_\Y$ to be a partition of $\C_\Y$ into cells of the form $T = K \times I$, where $K \in \T_\Omega$,
and $I$ is an interval. We consider the partition $\{ y_k\}_{k=0}^{M}$ of the interval $[0,\Y]$
\begin{equation}
\label{graded_mesh}
  y_k = \left( \frac{k}{M}\right)^{\mu} \Y, \quad k=0,\dots,M,
\end{equation}
where $\mu = \mu(\alpha) > 3/(1-\alpha)> 1$. Notice that each discretization of the truncated cylinder $\C_{\Y}$ depends on the truncation parameter $\Y$. The set of all such triangulations $\T_{\Y}$ is denoted by $\Tr$. In addition, if the partitions in the extended direction are given by \eqref{graded_mesh}, the following weak regularity condition is valid: there is a constant $\sigma$ such that, for all $\T_\Y \in \Tr$, if $T_1=K_1\times I_1,T_2=K_2\times I_2 \in \T_\Y$
have nonempty intersection, then
$
   h_{I_1}/h_{I_2} \leq \sigma,
$
where $h_I = |I|$; see \cite{DL:05,NOS}.

The main motivation to consider elements as in \eqref{graded_mesh} is to compensate the rather singular behavior of $\ve$, solution to problem \eqref{heat_harmonic_extension_weak}, as $y \downarrow0$. It is known that the numerical approximation of functions with a strong directional-dependent behavior needs anisotropic elements in order to recover
quasi-optimal error estimates \cite{DL:05,NOS2}. In our setting, anisotropic elements of tensor product structure are essential.

Given $\T_\Y$, we call $\N(\T_{\Y})$ the set of its nodes and $\Nin(\T_{\Y})$ the set of its interior and Neumann 
nodes. We denote by $N = \# \Nin(\T_{\Y})$ the number of degrees of freedom of $\T_\Y$. In what follows we assume that $\#\T_\Omega \approx M^n$ so that $N\approx M^{n+1}$. For each vertex $\vero \in \N$, we write $\vero= (\vero',\vero'')$, where $\vero'$ corresponds to a node of $\T_\Omega$, and $\vero''$ corresponds to a node of the partition of $[0,\Y]$. We define $h_{\vero'}= \min\{h_K: \vero' \textrm{ is a vertex of } K\}$, and $h_{\vero''} = \min\{h_I: \vero'' \textrm{ is a vertex of } I\}$. Given $\vero \in \Nin(\T_{\Y})$, we define the \emph{star} 
$
  S_{\vero} :=\bigcup_{T \ni \vero} T,
$
and for $T \in \T_{\Y}$ we set
$
  S_T := \bigcup_{\vero \in T} S_\vero.
$
For $\T_{\Y} \in \Tr$, we define 
\[
  \V(\T_\Y) := \left\{
            W \in C^0( \bar{\C}_\Y): W|_T \in \mathcal{P}_1(K)\otimes \mathbb{P}_1(I) \ \forall T = K \times I \in \T_\Y, \
            W|_{\Gamma_D} = 0
          \right\},
\]
where $\Gamma_D = \partial_L \C_{\Y} \cup \Omega \times \{ \Y\}$ is called the Dirichlet boundary. If $K$ is a simplex, then $\mathcal{P}(T)=\mathbb{P}_1(K)$, whereas if $K$ is a $n$-rectangle, then $\mathcal{P}(T)=\mathbb{Q}_1(K)$.
We also define $\U(\T_{\Omega}):=\tr \V(\T_{\Y})$, \ie a $\mathcal{P}_1$ finite element space over the mesh $\T_\Omega$.

The graded meshes described by \eqref{graded_mesh} yield near optimal error estimates both in regularity and order
for the elliptic case investigated in \cite{NOS}.

\subsection{Weighted elliptic projector: definition}
\label{sec:sub:elliptic_projector}

In this subsection, we define a \emph{weighted elliptic projector}, which is fundamental in \S~\ref{sec:fully_scheme}.
This projector is the operator
$
  G_{\T_{\Y}}: \HL(y^{\alpha},\C_\Y) \rightarrow \V(\T_\Y)
$
such that, for $w \in \HL(y^\alpha,\C_\Y)$, is given by
\begin{equation}
 \label{elliptic_projection}
a_\Y \left( G_{\T_{\Y}}w , W \right)  = a_\Y(w, W), \quad \forall W \in \V(\T_{\Y}).
\end{equation}

To easily describe the properties of the weighted elliptic projection operator $G_{\T_{\Y}}$ we introduce the 
mesh-size functions $h', h'' \in L^{\infty}(\C_{\Y})$ given by
\[
 h'_{|T} = h_K, \quad h''_{|T} = h_{I} \quad \forall T = K \times I \in \T_{\Y}.
\]
The operator $G_{\T_{\Y}}$ satisfies the following stability and approximation properties.
\begin{proposition}[weighted elliptic projector]
\label{pro:elliptic_P}
The weighted elliptic projector $G_{\T_{\Y}}$ is stable in $\HL(y^{\alpha},\C_\Y)$, \ie
\begin{equation}
 \label{P-stable}
\| \nabla G_{\T_{\Y}} w \|_{L^2(y^{\alpha},\C_\Y)} \lesssim \| \nabla w\|_{L^2(y^{\alpha},\C_\Y)}, \qquad
\forall w \in \HL(y^{\alpha},\C_\Y).
\end{equation}
If, in addition, $w \in H^2(y^{\alpha},\C_\Y)$, then $G_{\T_{\Y}}$ has the following approximation property
\begin{equation}
 \label{P-approx}
\| \nabla ( w - G_{\T_{\Y}} w) \|_{L^2(y^\alpha,\C_\Y)} \lesssim 
\| h' \nabla_{x'}  \nabla w\|_{L^2(y^{\alpha},\C_\Y)}
+
\| h'' \partial_{y}  \nabla w\|_{L^2(y^{\alpha},\C_\Y)}.
\end{equation}
\end{proposition}
\begin{proof}
To show stability set $W = G_{\T_\Y} w$ in \eqref{elliptic_projection}, use Cauchy-Schwarz inequality and the equivalence of $a_\Y(w,w)$ with $\| \nabla w\|^2_{L^2(y^{\alpha},\C_{\Y})}$ (see Remark~\ref{rem:equivalent}).

Obtaining the estimate \eqref{P-approx} hinges on Galerkin orthogonality, which yields
\begin{align*}
\| \nabla ( w - G_{\T_{\Y}} w) \|^2_{L^2(y^{\alpha},\C_\Y)}  
& \lesssim a_\Y \left( w - G_{\T_{\Y}} w,  w - \Pi_{\T_{\Y}} w \right)
\end{align*}
where $\Pi_{\T_{\Y}} $ is the interpolation operator defined in \cite{NOS}. The assertion then
follows from the anisotropic interpolation estimates of \cite[Theorems 4.7 and 4.8]{NOS}.
\end{proof}

\begin{lemma}[error estimates: elliptic projector]
\label{le:P_optimal}
If $w \in \HL(y^{\alpha},\C_\Y)$, $\mathcal{S}(w) < \infty$ and the mesh $\T_\Y$ is graded as in \eqref{graded_mesh}, then we have
\begin{equation}
\label{P_optimal}
\| \nabla ( w - G_{\T_{\Y}} w) \|_{L^2(y^\alpha,\C_\Y)} 
+ \| \tr ( w - G_{\T_{\Y}} w)  \|_{\Hs} \lesssim
 |\log N |^s N^{-1/(n+1)} \mathcal{S}(w),
\end{equation}
where $N = \# \T_{\Y}$ and $\mathcal{S}(w)$ is defined in \eqref{C(w)}.
\end{lemma}
\begin{proof}
The estimate for the first term is a direct consequence of \eqref{P-approx}, together with the fact that $\mathcal{S}(w) < \infty$ and \cite[Theorem 5.4]{NOS}, 
where the graded mesh \eqref{graded_mesh} on the extended variable $y$ is essential to recover near optimality. 
The bound for the second term is a consequence of the trace
estimate \eqref{Trace_estimate}.
\end{proof}

As with a standard, unweighted, elliptic projection we can obtain improved estimates for the weighted elliptic projection $G_{\T_\Y}$ in the $L^2(\Omega)$ norm via duality.

\begin{proposition}[$L^2(\Omega)$-approximation]
\label{pro:L2}
If $w \in \HL(y^{\alpha},\C_\Y)$, $\mathcal{S}(w) < \infty$ and the mesh $\T_\Y$ is graded as in \eqref{graded_mesh}, then we have
\begin{equation}
 \label{L2-approx-graded}
\|\tr( w - G_{\T_{\Y}} w)\|_{L^2(\Omega)} \lesssim |\log N|^{2s} N^{-\frac{1+s}{n+1}} \mathcal{S}(w),
\end{equation}
where $N = \# \T_{\Y}$ and $\mathcal{S}(w)$ is defined in \eqref{C(w)}.
\end{proposition}
\begin{proof}
Let $\mathcal{E} = w - G_{\T_{\Y}} w$, $e = \tr (w - G_{\T_{\Y}} w)$ and we denote by $P_{\T_{\Omega}}:L^2(\Omega) \rightarrow \U(\T_{\Omega})$ the standard $L^2$-projection. With this notation
$
 \|e\|_{L^2(\Omega)} \leq \|e - P_{\T_{\Omega}}e\|_{L^2(\Omega)} + \|P_{\T_{\Omega}}e\|_{L^2(\Omega)}.
$
The estimate of the first term follows from standard polynomial interpolation and Hilbert space interpolation arguments
\[
 \|e - P_{\T_{\Omega}}e \|_{L^2(\Omega)} \lesssim h_{\T_{\Omega}}^{s} \| e \|_{\Hs}
\lesssim h_{\T_{\Omega}}^{s} \| \nabla \mathcal E \|_{L^2(y^{\alpha},\C_{\Y})} \lesssim \mathcal{S}(w)|\log N|^{s} N^{-\frac{1+s}{n+1}}.
\]
To estimate the remaining term we argue by duality. Let $z \in  \HL(y^{\alpha},\C_{\Y})$ solve 
\begin{equation}
\label{dual_problem}
a_\Y(\phi, z) = \langle P_{\T_{\Omega}}e, \tr \phi \rangle, \quad
\forall  \phi \in \HL(y^{\alpha},\C_{\Y}).
\end{equation}
Set $\phi = \mathcal E$. Using the definition of $P_{\T_{\Omega}}$, that $e = \tr \mathcal{E}$ and \eqref{dual_problem}, we obtain
\[
\| P_{\T_{\Omega}}e \|^2_{L^2(\Omega)} = a_{\Y}(\mathcal{E},z) \lesssim \| \nabla (w-G_{\T_{\Y}}w)\|_{L^2(y^{\alpha},\C_{\Y})} \| \nabla (z - G_{\T_{\Y}}z)\|_{L^2(y^{\alpha},\C_{\Y})}.
\]
Applying Lemma \ref{le:P_optimal} to $z$, in conjunction with $\mathcal{S}(z) \lesssim \|P_{\T_{\Omega}}e \|_{\Ws}$ \cite[Theorem 2.7]{NOS} for $z$, we arrive at
\[
\| \nabla (z - G_{\T_{\Y}}z)\|_{L^2(y^{\alpha},\C_{\Y})} 
\lesssim
|\log N|^{s}N^{-\frac{1}{n+1}} \mathcal{S}(z) \lesssim |\log N|^{s}N^{-\frac{1}{n+1}}\|P_{\T_{\Omega}}e \|_{\Ws}.
\]
The inverse estimate 
$\| P_{\T_{\Omega}}e \|_{\Ws} \lesssim h^{s-1}_{\T_{\Omega}} \| P_{\T_{\Omega}}e \|_{L^2(\Omega)}$
and Lemma~\ref{le:P_optimal} yield
\[
\| P_{\T_{\Omega}} e \|_{L^2(\Omega)} \lesssim \mathcal{S}(w) 
|\log N|^{2s} N^{-\frac{1+s}{n+1}},
\]
which implies the asserted estimate \eqref{L2-approx-graded}.
\end{proof}

\section{A fully discrete scheme for $\gamma \in (0,1]$}
\label{sec:fully_scheme}

Let us now describe a fully discrete numerical scheme to solve \eqref{heat_harmonic_extension_weak_truncated}. The space discretization hinges on the finite element method on a truncated cylinder discussed in \S\ref{sec:space_discretization}. The discretization in time uses the implicit finite difference
schemes proposed in \S\ref{sub:discretization_1} for $\gamma =1$ and in \S\ref{sub:discretization_beta} for $\gamma \in (0,1)$.

The fully discrete scheme computes the sequence $V_{\T_{\Y}}^\tau  \subset \V(\T_{\Y})$, an approximation of the solution to problem \eqref{heat_harmonic_extension_weak_truncated} at each time step. We initialize the scheme by setting
\begin{equation}
\label{initial_data_discrete}
V_{\T_{\Y}}^{0} = \mathcal I_{\T_\Omega}\usf_0,
\end{equation}
where $\mathcal{I}_{\T_\Omega} = G_{\T_\Y} \circ \mathcal{H}_\alpha$ and $\mathcal{H}_\alpha$ is the $\alpha$-harmonic extension onto $\C_\Y$ (see Remark~\ref{rem:initialdata}); notice that $\tr V_{\T_\Y}^0 = \tr G_{\T_\Y} v(0)$.
For $k=0,\dots,\K-1$, let $V_{\T_{\Y}}^{k+1} \in \V(\T_{\Y})$ solve
 \begin{equation}
 \label{fully_beta}
 ( \delta^\gamma \tr V_{\T_{\Y}}^{k+1} , \tr W )_{L^2(\Omega)}  +
a_\Y(V_{\T_{\Y}}^{k+1},W) = \left\langle f^{k+1}, \tr W   \right\rangle, \quad \forall W \in \V(\T_\Y).
 \end{equation}
The discrete operator $\delta^\gamma$ is defined in \eqref{L_discrete} for $\gamma \in (0,1)$ and in \eqref{partial} for $\gamma = 1$. An approximate solution to problem \eqref{fractional_heat} is given by the sequence
$U_{\T_{\Omega}}^\tau  \subset \V(\T_\Omega)$:
\begin{equation}
\label{discrete_Ufd}
 U_{\T_{\Omega}}^\tau = \tr V_{\T_{\Y}}^\tau.
\end{equation}
As before, \eqref{initial_data_discrete}--\eqref{fully_beta} is a discrete elliptic problem with dynamic boundary condition. 


We have the following unconditional stability result.

\begin{lemma}[unconditional stability]
\label{le:stab_beta_full}
The discrete scheme \eqref{initial_data_discrete}--\eqref{fully_beta} is unconditionally stable 
for all $\gamma \in (0,1]$, \ie
\begin{equation}
  I^{1-\gamma}\|\tr V_{\T_{\Y}}^\tau \|^2_{L^2(\Omega)}(T) + \| V_{\T_{\Y}}^{\tau} \|^2_{\ell^2(\HLn(y^{\alpha},\C_{\Y}))}
  \lesssim \Lambda_\gamma( V^0_{\T_\Y}, f^\tau )^2,
\label{stab_beta_fully} 
\end{equation}
where $I^0$ is the identity according to Remark~\ref{RE:militing_case} (case $\gamma=1$).
\end{lemma}
\begin{proof}
Set $W = 2 \tau V_{\T_{\Y}}^{k+1}$ for $\gamma = 1$ and $W = 2 \Gamma(2-\gamma)\tau^{\gamma} V_{\T_{\Y}}^{k+1}$ for $ 0 < \gamma <1 $ in \eqref{fully_beta} and proceed as in Lemma~\ref{le:stab_1} and Theorem \ref{TH:stabbeta}, respectively.
\end{proof}

Let us now obtain an error estimate for the fully discrete scheme \eqref{fully_beta}.
We split the error into the so-called interpolation and approximation errors \cite{Guermond-Ern,Thomee}: 
\[
v^{\tau} - V_{\T_{\Y}}^\tau  = \left( v^{\tau}- G_{\T_{\Y}}v^{\tau} \right) + ( G_{\T_{\Y}}v^{\tau} - V_{\T_{\Y}}^\tau )
= \eta^{\tau} + E_{\T_{\Y}}^\tau.
\]
Property \eqref{P_optimal} implies that $\eta$ is controlled near-optimally in energy
\begin{equation}
\| \nabla \eta^{\tau} \|_{\ell^2(L^2(y^{\alpha},\C_\Y))} \lesssim  |\log N|^s N^{-1/(n+1)} \| \mathcal{S}(v^{\tau}) \|_{L^2(0,T)}.
\end{equation}
Estimate \eqref{eq:S_in_spacetime}, Corollary~\ref{CR:Sut} and Remark~\ref{rem:reg_of_v} imply that $\mathcal{S}(v) \in W_1^1(0,T)$, whence
\begin{equation}
\label{eta_td}
\| \nabla \eta^{\tau} \|_{\ell^2(L^2(y^{\alpha},\C_\Y))} \lesssim |\log N|^s N^{\frac{-1}{n+1}} \mathcal{B}(\usf_0,f),
\end{equation}
since $\mathcal{R}(\usf_0,f) \leq \mathcal{B}(\usf_0,f)$. Similar arguments, together with \eqref{L2-approx-graded}, allow us to conclude an approximation result in the $L^2$-norm for the trace
\begin{align}
\label{tr_etad}
I^{1-\gamma}\| \tr\eta^{\tau} \|_{L^2(\Omega)}(T) \lesssim |\log N|^{2s} N^{-\frac{1+s}{n+1}} 
I^{1-\gamma} \mathcal{B}(\usf_0,f)(T).
\end{align}
The error estimates for \eqref{initial_data_discrete}--\eqref{fully_beta} read as follows.

\begin{theorem}[error estimates: $\gamma \in (0,1)$]
\label{th:order_beta}
Let $\gamma \in (0,1)$, $v$ and $V_{\T_{\Y}}^\tau$ solve \eqref{heat_harmonic_extension_weak_truncated} and \eqref{initial_data_discrete}--\eqref{fully_beta}, respectively. If $\mathcal{A}(\usf_0,f), \mathcal{B}(\usf_0,f) < \infty$ and $\T_{\Y}$ verifies \eqref{graded_mesh}, then
\begin{equation}\label{estimate_1}
[ I^{1-\gamma}\| \tr (v^\tau - V_{\T_{\Y}}^\tau) \|^2_{L^2(\Omega)}(T)]^{\tfrac{1}{2}} \lesssim 
\tau^{\theta} \mathcal{A}(\usf_0,f) + |\log N|^{2s} N^{-\frac{1+s}{n+1}} \mathcal{B}(\usf_0,f),
\end{equation}
and
\begin{equation}\label{estimate_2}
\| v^\tau - V_{\T_{\Y}}^\tau \|_{\ell^{2}(\HLn(y^{\alpha},\C_\Y) )} \lesssim \tau^{\theta} \mathcal{A}(\usf_0,f) 
  + |\log N|^{s} N^\frac{-1}{n+1} \mathcal{B}(\usf_0,f),
\end{equation}
where $\mathcal{A}$ and $\mathcal{B}$ are defined in \eqref{A} and \eqref{eq:frakB}, respectively, $0<\theta<\tfrac{1}{2}$, and the hidden constants blow up as $\theta \uparrow \tfrac12$.
\end{theorem}
\begin{proof}
Using the continuous problem \eqref{heat_harmonic_extension_weak_truncated}, the discrete equation \eqref{fully_beta}, and
the definition \eqref{elliptic_projection} of $G_{\T_{\Y}}$, we arrive at the equation that controls the error:
\begin{equation}
\label{discrete_eq_e_beta}
(\delta^\gamma \tr E_{\T_{\Y}}^{k+1}, \tr W )_{L^2(\Omega)}  + a_\Y(E_{\T_{\Y}}^{k+1},W) = \left\langle \tr \omega^{k+1}, \tr W   \right\rangle \quad  W \in \V(\T_{\Y}),
\end{equation}
where
$
  \omega^{k+1} = \delta^\gamma G_{\T_{\Y}}v(t_{k+1}) - \partial^{\gamma}_t v(t_{k+1}).
$
Estimate \eqref{stab_beta_fully} applied to \eqref{discrete_eq_e_beta} yields
\[
I^{1-\gamma}\| \tr E_{\T_{\Y}}^{\tau} \|^2_{L^2(\Omega)}(T) + \| E_{\T_{\Y}}^{\tau} \|^2_{\ell^2(\HLn(y^{\alpha},\C_{\Y}))} \lesssim \| \tr \omega^\tau \|_{\ell^2(\Hsd)}^2,
\]
because $\tr E^0_{\T_{\Y}} = 0$. We decompose $\omega^{k+1}$ as $\omega^{k+1} = \omega_1^{k+1} + \omega_2^{k+1}$
with
%
\[
\omega_1^{k+1} := \left( \delta^\gamma v(t_{k+1}) - \partial^{\gamma}_t v(t_{k+1}) \right),
\qquad \omega_2^{k+1} := \delta^\gamma \left( G_{\T_{\Y}}v(t_{k+1}) - v(t_{k+1}) \right).
\]
The first term is controlled using Proposition~\ref{pro:consistency}. For $\theta \in (0,\tfrac{1}{2})$ we have
\[
\| \tr \omega_1^\tau\|_{\ell^2(\Hsd)} \lesssim \tau^{\theta}\mathcal{A}(\usf_0,f),
\]
with a hidden constant that blows up as $\theta \uparrow \tfrac12$. To estimate $\omega_2^{k+1}$ we use \eqref{discretization_fractional} and \eqref{L_discrete} to write
\begin{align*}
\omega_2^{k+1} 
 = \frac{1}{\Gamma(2-\gamma)}\sum_{j=0}^{k}
\frac{a_j}{ \tau^{\gamma}} \int_{I_{k-j}} ( I - G_{\T_{\Y}} ) \partial_{t} v(s) \diff s,
\end{align*}
and use Proposition \ref{pro:L2} together with $\| \tr \omega_2^{k+1}\|_{\Hsd} \lesssim \| \tr \omega_2^{k+1}\|_{L^2(\Omega)}$ to obtain
\begin{equation*}
\| \tr \omega_2^{k+1} \|_{\Hsd} \lesssim \frac{\tau^{1-\gamma}}{\Gamma(2-\gamma)}
|\log N |^{2s} N^{-\frac{1+s}{n+1}} \sum_{j=0}^k a_j \fint_{I_{k-j}} \mathcal{S}( \partial_t v(s) ) \diff s.
\end{equation*}
If $Z^{\tau} := \big\{ \fint_{I_j} \mathcal{S}( \partial_t v(s) ) \diff s \big\}_{j=0}^{\K-1}$, then the definition of the fractional integral \eqref{fractional_integral} in conjunction with \eqref{discretization_fractional} implies
\[
 \| \tr \omega_2^{k+1}\|_{\Hsd} \lesssim  |\log N|^{2s} N^\frac{-(1+s)}{n+1}I^{1-\gamma} Z^{\tau}(t_{k+1}).
\]
We recall that, according to \eqref{Sut} and Remark \ref{rem:reg_of_v}, $\mathcal{S}(\partial_t v(s) ) \lesssim s^{\gamma-1} \mathcal{B}(\usf_0,f)$. We argue with $Z^{\tau}$ as in Corollary~\ref{CR:Sut} to obtain
\[
\| I^{1-\gamma} Z^{\tau}\|_{L^2(0,T)} \lesssim \| Z^{\tau} \|_{\Xcal}
\]
where $\Xcal = L^2(0,T)$ if $\gamma \in (\frac{1}{2},1)$ and $\Xcal = L \log L (0,T)$ if $\gamma \in (0,\frac{1}{2}]$.
We next use that local averages are continuous in $\Xcal$ to deduce
\[
  \|\tr \omega_2^\tau \|_{\ell^2(\Hsd)} \lesssim 
  |\log N |^{2s} N^{-\frac{1+s}{n+1}} \mathcal{B}(\usf_0,f).
\]
Collecting all the previous estimates together with 
\eqref{eta_td} and \eqref{tr_etad}, allows us to obtain the desired results.
\end{proof}

\begin{remark}[estimate for $\usf$: $\gamma \in (0,1)$] \rm
In the framework of Theorem~\ref{th:order_beta}, and in view of \eqref{exp_convergence}, we deduce the following error estimates for $\usf$
\begin{align*}
  \left[ I^{1-\gamma}\| \usf^{\tau} - U^{\tau} \|^2_{L^2(\Omega)}(T) \right]^{\frac{1}{2}}  &\lesssim 
    \tau^\theta \mathcal{A}(\usf_0,f)
    + |\log N |^{2s} N^\frac{-(1+s)}{n+1} \mathcal{B}(\usf_0,f) \\
    & + e^{-\frac{\sqrt{\lambda_1}}{2}\Y} \Lambda_\gamma(\usf_0,f), \\
  \| \usf^{\tau} - U^{\tau} \|_{\ell^{2}(\Hs )}  \lesssim \tau^\theta \mathcal{A}(\usf_0,f)
  &+ |\log N|^{s} N^\frac{-1}{n+1} \mathcal{B}(\usf_0,f) + e^{-\frac{\sqrt{\lambda_1}}{2}\Y } \Lambda_\gamma(\usf_0,f),
\end{align*}
where $0 <\theta < \tfrac{1}{2}$ and $\mathcal{A}$, $\mathcal{B}$ and $\Lambda_\gamma$ are defined by \eqref{A}, \eqref{eq:frakB} and \eqref{Lambda}, respectively.
\end{remark}

To conclude we establish error estimates for $\gamma = 1$. Denote
\[
  C(\usf_0, f) = \| \usf_0 \|_{\mathbb{H}^{2s}(\Omega)} + \| f \|_{BV(0,T;L^2(\Omega))}.
\]
The estimates read as follows.
\begin{theorem}[error estimates: $\gamma = 1$]
Let $\gamma = 1$, $v$ and $V_{\T_{\Y}}^\tau$ solve 
\eqref{heat_harmonic_extension_weak_truncated} and
\eqref{initial_data_discrete}--\eqref{fully_beta}, respectively. 
If $\T_{\Y}$ is graded according to \eqref{graded_mesh}, then 
\begin{align*}
  \| \tr (v^\tau - V_{\T_{\Y}}^\tau) \|_{\ell^{\infty}(L^2(\Omega))}  &\lesssim \tau C(\usf_0,f) 
  + |\log N|^{2s} N^\frac{-(1+s)}{n+1} \| \mathcal{S}(v_t) \|_{W_1^1(0,T)}, \\
  \| v^\tau - V_{\T_{\Y}}^\tau \|_{\ell^{2}(\HLn(y^{\alpha},\C_\Y) )}  &\lesssim \tau C(\usf_0 f)
  +|\log N |^{s} N^\frac{-1}{n+1} \| \mathcal{S}(v_t) \|_{W_1^1(0,T)},
\end{align*} 
where the hidden constants are independent of the data, $N$ and $\tau$.
\end{theorem}
\begin{proof}
The proof is standard and relies on the arguments developed in
Theorem~\ref{thm:semidicserr}, Theorem~\ref{th:order_beta} and \cite[Theorem 3.20]{NSV:00}.
\end{proof}

\section*{Acknowledgments}
We would like to thank W. McLean for pointing out a flaw in our original manuscript, and
the two referees for several comments and suggestions that led to much better results and improved presentation.

\bibliographystyle{plain}
\bibliography{biblio}

\end{document}